\documentclass[12pt,oneside]{amsart}
\usepackage{ae,aecompl}

\usepackage[english]{babel}

\usepackage{amsthm}
\usepackage[dvips]{graphicx}
\usepackage{amssymb}

\numberwithin{equation}{section}
\numberwithin{figure}{section}

\theoremstyle{plain}
\newtheorem{thm}{Theorem}[section]
\newtheorem{lem}[thm]{Lemma}
\newtheorem{cor}[thm]{Corollary}
\newtheorem{prop}[thm]{Proposition}
\theoremstyle{definition}
\newtheorem{defn}[thm]{Definition}

\newcommand{\Pf}{\mathrm{Pf}}

\usepackage[all, dvips]{xy}

\begin{document}

\title{Almost commuting unitary matrices related to time reversal}

\author{Terry A. Loring and Adam P. W. S\o rensen }

\address{University of New Mexico, Department of Mathematics and Statistics,
Albuquerque, New Mexico, 87131, USA}

\address{Department of Mathematical Sciences, University of Copenhagen,
Universitetsparken 5, DK-2100, Copenhagen \O , Denmark}

\begin{abstract}
The behavior of fermionic systems depends on the geometry of the system
and the symmetry class of the Hamiltonian and observables. Almost
commuting matrices arise from band-projected position observables
in such systems. One expects the mathematical behavior of almost commuting
Hermitian matrices to depend on two factors. One factor will be the
approximate polynomial relations satisfied by the matrices. The other
factor is what algebra the matrices are in, either $\mathbf{M}_{n}(\mathbb{A})$
for $\mathbb{A}=\mathbb{R}$, $\mathbb{A}=\mathbb{C}$ or $\mathbb{A}=\mathbb{H}$,
the algebra of quaternions. 

There are potential obstructions keeping $k$-tuples of almost commuting operators
from being close to a commuting $k$-tuple. We consider two-dimensional
geometries and so this obstruction lives in $KO_{-2}(\mathbb{A})$.
This obstruction corresponds to either the Chern number or spin Chern
number in physics. We show that if this obstruction is the trivial
element in $K$-theory then the approximation by commuting matrices
is possible.
\end{abstract}

\maketitle
\markright{Almost commuting unitary matrices related to time reversal}

\section{Introduction}

\subsection{Approximate representations }

Consider the two sets of relations
\[
\framebox[2.3in]{\ensuremath{
\mathcal{S}_{\delta}
\begin{array}[t]{c}
 X_{r}^{*}=X_{r}\\
\left\Vert X_{r}X_{s}-X_{s}X_{r}\right\Vert \leq\delta\\
\left\Vert X_{1}^{2}+X_{2}^{2}+X_{3}^{2}-I\right\Vert \leq\delta\\
\strut\end{array}
}}
\ 
\framebox[2.0in]{\ensuremath{
\mathcal{T}_{\delta}^{\prime}
\ 
\begin{array}[t]{c}
 X_{r}^{*}=X_{r}\\
\left\Vert X_{r}X_{s}-X_{s}X_{r}\right\Vert \leq\delta\\
\left\Vert X_{1}^{2}+X_{2}^{2}-I\right\Vert \leq\delta\\
\left\Vert X_{3}^{2}+X_{4}^{2}-I\right\Vert \leq\delta
\end{array}
}}
\]
that we call the soft sphere relations, $\mathcal{S}_{\delta}$, in
matrix unknowns $X_{1}$, $X_{2}$, $X_{3}$, and the soft torus relations,
$\mathcal{T}_{\delta}^{\prime}$, in matrix unknowns $X_{1}$, $X_{2}$,
$X_{3}$, $X_{4}$. Our main results concern the operator-norm distance
from a representation of $\mathcal{S}_{\delta}$ or
$\mathcal{T}_{\delta}^{\prime}$ to representations of
$\mathcal{S}_{0}$ or $\mathcal{T}_{0}^{\prime}$.  We show that
the distance one must move away from the $X_{r}$
to find Hermitian matrices that are commuting goes to zero as
the commutator norm goes to zero.

We show that when the $X_{r}$ are taken as variables in
$\mathbf{M}_{n}(\mathbb{R})$, then representations of
$\mathcal{S}_{\delta}$ are, in a uniform fashion, close
to representations of $\mathcal{S}_{0}$. When the
$X_{r}$ are taken as variables in $\mathbf{M}_{n}(\mathbb{C})$,
there is an obstruction in $\mathbb{Z}$ that dictates the possibility
of such an approximation. When the $X_{r}$ are taken as variables
in $\mathbf{M}_{n}(\mathbb{H})$, (the algebra of quaternions) there
is again an obstruction, but now in $\mathbb{Z}_{2}$. For
$X_{r}\in\mathbf{M}_{n}(\mathbb{A})$ in these three cases,
the obstruction is formally defined to be in $KO_{-2}(\mathbb{A})$
and we prove this is the only obstruction to the desired
approximation by commuting Hermitian matrices. 

The complex case of our results were proven in \cite[Corollary 13]{LoringWhenMatricesCommute}
and \cite[Corollary 6.15]{ELP-pushBusby}, using the techniques of semiprojectivity
and $K$-theory for $C^{*}$-algebras. The connection of these matrix
results to condensed matter physics was not noticed until many years
later \cite{hastings2008topology,HastingsLoringWannier,HastLorTheoryPractice,LorHastHgTe}.
Of course the relevance of the $K$-theory of $C^{*}$-algebras to
condensed matter physics was known earlier \cite{BellissardNCGquantumHall}.

\subsection{Structured complex matrices}

Quaternionic matrices arrive in disguise in phsyics via the isometric
emdedding
$\chi \colon \mathbf{M}_{N}(\mathbb{H})\rightarrow\mathbf{M}_{2N}(\mathbb{C})$
defined as
\[
\chi\left(A+B\hat{j}\right)=\left[\begin{array}{cc}
A & B\\
-\overline{B} & \overline{A}\end{array}\right]
\]
for complex matrices $A$ and $B$. The image of $\chi$ can be described
as the matrices that commute with the antiunitary operator
\[
\mathcal{T}\xi=-Z\overline{\xi}
\]
where $\xi$ is in $\mathbb{C}^{2N}$ and
\[
Z=Z_{N}=\begin{bmatrix}0 & I\\
-I & 0\end{bmatrix}.
\]

In a finite model, $\mathcal{T}$ is typically playing the role of
time-reversal. From a purely mathematical standpoint, it is generally
easier to think in terms of the \emph{dual operation}
\begin{equation}
\left[\begin{array}{cc}
A & B\\
C & D\end{array}\right]^{\sharp}=\left[\begin{array}{cc}
D^{\mathrm{T}} & -B^{\mathrm{T}}\\
-C^{\mathrm{T}} & A^{\mathrm{T}}\end{array}\right]
\label{eq:dual-def-1}
\end{equation}
alternatively defined as
\begin{equation}
X^{\sharp}=-ZX^{\mathrm{T}}Z.
\label{eq:dual-def-2}
\end{equation}
The image of $\chi$ is the set of matrices with $X^{*}=X^{\sharp}$.
We are using mathematical notation, so $X^{*}$ refers to the
conjugate-transpose of $X$.

Similarly, we think of a real matrix $X$ as a complex matrix for
which $X^{*}=X^{\mathrm{T}}$. 

A good survey paper regarding the equivalence of matrices of quaternions
and structured complex matrices is \cite{FarenickSpectralThmquatern}.
This does not address norms on $\mathbf{M}_{N}(\mathbb{H})$.

The norm we consider here is that induced by the operator norm on
$\mathbf{M}_{2N}(\mathbb{C})$. Thus we simply use $\left\Vert A\right\Vert $
to denote the operator norm (a.k.a.~the spectral norm) of a complex
matrix and define for $X$ in $\mathbf{M}_{N}(\mathbb{H})$ the norm
to be
\[
\left\Vert X\right\Vert =\left\Vert \chi\left(X\right)\right\Vert .
\]
This norm can be seen to be the norm induced by the action of $X$
on $\mathbb{H}^{N}$, but that is not important. In the same way,
we could think about the norm of a real matrix induced by its action
on $\mathbb{R}^{n}$ but prefer to consider this norm as being defined
via its action on $\mathbb{C}^{n}$. 

All theorems will be stated in terms of complex matrices, possibly
respecting an additional symmetry such as being self-dual. This allows
us the freedom to combine two Hermitian matrices $X_{1}$ and $X_{2}$
into one matrix \[
U=X_{1}+iX_{2}.\]
If $X_{1}$ and $X_{2}$ are self-dual, then they are in
$\chi(\mathbf{M}_{N}(\mathbb{H}))$, whereas $U$ is self-dual
but most likely $U^{*}\neq U^{\sharp}$.  That is, $U$ cannot be
assumed to be in $\chi(\mathbf{M}_{N}(\mathbb{H}))$.
We abandon $\mathcal{T}_{\delta}^{\prime}$ in favor of the relations
$\mathcal{T}_{\delta}$ 
\[
\framebox[2.0in]{\ensuremath{\mathcal{T}_{\delta}}\ \ensuremath{\begin{array}[t]{c}
 U_{r}^{*}U_{r}=U_{r}U_{r}^{*}=I\\
\left\Vert U_{1}U_{2}-U_{2}U_{1}\right\Vert \leq\delta\end{array}}}
\]
which we no longer apply elements of $\mathbf{M}_{N}(\mathbb{H})$.
(They \emph{can} be applied there, but doing so leads to different
question that have less obvious connections to physics.) Instead we
apply them to complex matrices $U_{1}$ and $U_{2}$ and then add
as appropriate $U_{r}^{\sharp}=U_{r}$ or $U_{r}^{\mathrm{T}}=U_{r}$. 

\emph{Beyond this point, matrices are assumed to be complex.}

\subsection{Real $C^{*}$-algebras}

Most of our theorems are statements about Real $C^{*}$-algebras.
More specifically, we consider $C^{*,\tau}$-algebras. This is an
ordinary (so complex) $C^{*}$-algebra $A$ with the additional structure
of an \emph{anti-multiplicative}, $\mathbb{C}$-linear map $\tau \colon A\rightarrow A$
for which $\tau(\tau(a))=a$ and $\tau(a^{*})=\tau(a)^{*}$. 

We prefer the notation $a\mapsto a^{\tau}$ to keep close to our essential
examples $\left(\mathbf{M}_{n}(\mathbb{C}),\mathrm{T}\right)$ and
$\left(\mathbf{M}_{n}(\mathbb{C}),\sharp\right)$. For background
of this perspective see \cite{HastLorTheoryPractice}. For a reference
that uses more traditional notation, see \cite{li2003real}. Most
importantly, our proofs rely on many results from our previous paper
\cite{LorSorensenDisk}. This deals with the same approximation-by-commuting
question, but with the equations $\mathcal{D}_{\delta}$ 
\[
\framebox[2.3in]{\ensuremath{
\mathcal{D}_{\delta}
\ \ \ \ 
\begin{array}[t]{c}
 X_{r}^{*}=X_{r}\\
\left\Vert X_{1}X_{2}-X_{2}X_{1}\right\Vert \leq\delta\\
\left\Vert X_{1} \right\Vert, \left \Vert X_{2} \right\Vert \leq 1
\end{array}
}}
\]
in two matrix variables. The underlying geometry is the disk and so
the potential for an obstruction in $K$-theory is eliminated. Hermitian
almost commuting matrices are always close to commuting Hermitian
matrices. For complex matrices, this is Lin's theorem
\cite{LinAlmostCommutingHermitian}, and the result stays true in
the real or self-dual case.

Our approach is to move from the disk to the sphere along the lines
of \cite{LoringWhenMatricesCommute}, this involves reformulating
the problem as a lifting problem. 
Then we move from the sphere to
the torus using the interplay between push-out diagrams and extensions,
generalizing results from \cite{ELP-pushBusby}.

\subsection{Band-projected position matrices}

Consider a lattice model with Hamiltonian $H=H^{*}$, for a single
particle, tight-binding model on a surface in $d$-space determined
by some equations $p(x_{1},\dots,x_{d})=0$. The position matrices
in the model will be diagonal, so matrices $\hat{X}_{r}$
with $[\hat{X}_{r},\hat{X}_{s}]=0$ and $p(\hat{X}_{1},\dots,\hat{X}_{d})=0$.
Under some assumptions, essentially that we have a spectral gap and
local interactions, the Hamiltonian will approximately commute with
the position matrices. Let $P$ denote the projection onto the states
below the Fermi energy. When we form the band-projected position matrices
$X_{r}=P\hat{X}_{r}P$ we arrive at matrices that are only almost
commuting. Depending on the universality class \cite{mehta2004random}
of the system, both $H$ and $\hat{X}_{r}$ will have extra symmetries,
as will the $X_{r}$. For example, time reversal invariance will result
in these matrices being self-dual.

There are important details (see \cite{HastLorTheoryPractice}) needed to
correct for the fact that the $X_{r}$ do not have full rank and that
the equations need to change to allow for larger physical size of
the lattice when the number of lattice size increases. 
In many interesting cases the result is matrices $X_{1}^{(n)},\dots,X_{d}^{(n)}$ of increasing size with
\[
\left\Vert \left[X_{r}^{(n)},X_{s}^{(n)}\right]\right\Vert
\rightarrow 0
\]
and
\[
\left\Vert p(X_{1},\dots,X_{d})\right\Vert \rightarrow 0.
\]
When the lattice geometry is the two-torus, the resulting obstruction
to the matrix approximation problem is an integer that corresponds
to the Chern number. When time reversal invariance is assumed, the
extra symmetry leads to an obstruction in $\mathbb{Z}_{2}$. This
obstruction, for the self-dual matrix approximation problem, corresponds
to the spin Chern number used to detect two-dimensional topological
insulators. Changing the geometry to a three-torus leads to an obstruction
in $KO_{-3}(\mathbb{H})\cong\mathbb{Z}_{2}$. There is numerical evidence
\cite{HastLorTheoryPractice}, and the $K$-homology arguments of
\cite{kitaev-2009} and \cite[\S III]{freedman2011projective}, that
this obstruction will be useful for detecting three-dimensional topological
insulators.

\subsection{The obstructions}

When $X$ is any invertible matrix, we define $\mathrm{polar}(X)$ as the unitary in the polar decomposition.
That is, it is the {\em polar part} of $X$ or, in terms of the functional calculus,
\[
\mathrm{polar}(X) = X\left( X^* X \right) ^{-\frac{1}{2}} .
\]

Consider $H_{1}$, $H_{2}$ and $H_{3}$ that are a representation
of $\mathcal{S}_{\delta}$ with $\delta<\tfrac{1}{4}$. We define
\[
B(H_{1},H_{2},H_{3})=\left[\begin{array}{cc}
H_{3} & H_{1}+iH_{2}\\
H_{1}-iH_{2} & -H_{3}\end{array}\right],
\]
which will be Hermitian and invertible \cite[Lemma 3.2]{Loring-K-thryAsymCommMatrices}.
The obstruction to being close to commuting is defined as the $K$-theory class
\[
	\mathrm{Bott}(H_{1},H_{2},H_{3}) = \left[\strut\mathrm{polar}\left(B(H_{1},H_{2},H)\right)\right]\in K_{0}(\mathbb{C}).
\]
Here we adopt the convention that $K_{0}(A)$ is defined via homotopy
classes of self-adjoint unitary elements in $\mathbf{M}_{n}(A)$.
This is not the standard view in terms of projections, but is equivalent
by a simple shift and rescaling. This is for $A$ a unital $C^{*}$-algebra.

A more computable description is to use the signature, meaning half
of the number of positive eigenvalues of $B$ minus half the number
of negative eigenvalues. We call $\mathrm{Bott}(H_{1},H_{2},H_{3})$
the \emph{Bott index}.

(We are off by a minus sign from the definition in \cite{HastingsLoringWannier}.
The work in \cite{Loring-K-thryAsymCommMatrices} was done without
noticing the role of the Pauli spin matrices.)

Given a $C^{*,\tau}$-algebra $\left(A,\tau\right)$ we regard $K_{2}(A,\tau)$
as defined via classes of invertible in $\mathbf{M}_{2n}(A)$
with $x^* = x$ and $x^\tau = -x$.  In \cite{HastLorTheoryPractice} we worked
with inveribles with $x^* = -x$ and $x^\tau = -x$.  These are equivalent
and the conversion from one picture to the other is done simply by multiplying
by $i$.

When $H_{r}=H_{r}^{\sharp}$ for all $r$, or $H_{r}=H_{r}^{\mathrm{T}}$,
the Bott index vanishes so it is possible to approximate by commuting
Hermitian matrices. Notice these nearby commuting matrices will be
only approximately self-dual or symmetric. In the self-dual case a
new obstruction arises when we try to approximate by matrices that
are at once commuting, Hermitian and self-dual. Here the larger matrix
$B(H_{1},H_{2},H_{3})$ satisfies
\[
\left(B(H_{1},H_{2},H_{3})\right)^{\sharp\otimes\sharp}=B(H_{1},H_{2},H_{3}).
\]
Noticing that 
\[
K_{2}\left(\mathbf{M}_{4N}(\mathbb{C}),\sharp\otimes\sharp\right)\cong\mathbb{Z}_{2},
\]
the first named author and Hastings defined in \cite{HastLorTheoryPractice} the
\emph{Pfaffian-Bott index,} denoted 
\[
\mathrm{Pf}\!\mathrm{-}\!\mathrm{Bott}(H_{1},H_{2},H_{3}),
\]
as the $K$-theory class 
\[
\left[\strut\mathrm{polar}\left(B(H_{1},H_{2},H_{3})\right)\right]
\in
K_{2}(\mathbf{M}_{4N}(\mathbb{C}),\sharp\otimes\sharp)\cong\mathbb{Z}_{2}.
\]

Much of the work in \cite{LorHastHgTe} was in demonstrating that
the Pfaffian-Bott index can be efficiently computed numerically using
a Pfaffian. The Pfaffian cannot be applied directly, but conjugating
$B(H_{1},H_{2},H_{3})$ by a fixed unitary leads to a purely imaginary,
skew-symmetric matrix. 
The sign of the Pfaffian of that matrix tells us which $K_{2}$-class contains $\mathrm{polar}(B(H_{1},H_{2},H_{3}))$.

\subsection{Main theorems\label{sub:Main-theorems}}

We state now two of our four main theorems, along with the complex
version. For consistency with \cite{LorHastHgTe}, we regard the Bott
index as an element of $\mathbb{Z}$ and the Pfaffian-Bott index as
an element of the multiplicative group $\{\pm1\}$. 

\begin{thm}
\label{thm:SphereComplex}
{\rm (}\cite[Corollary 13]{LoringWhenMatricesCommute}{\rm )}
For every $\epsilon>0$ there exists a $\delta>0$ so that whenever
matrices $H_{1},H_{2},H_{3}$ form a representation of
$\mathcal{S}_{\delta}$, and
\[
\mathrm{Bott}(H_{1},H_{2},H_{3})=0,
\]
there are matrices $ $$K_{1},K_{2},K_{3}$ that form a representation
of $\mathcal{S}_{0}$ and so that
\[
\left\Vert K_{r}-H_{r}\right\Vert \leq\epsilon\quad(r=1,2,3).
\]

\end{thm}

\begin{thm}
\label{thm:SphereReal}
For every $\epsilon>0$ there exists a $\delta>0$
so that whenever $H_{1},H_{2},H_{3}$ are complex symmetric matrices
that form a representation of $\mathcal{S}_{\delta}$, there are complex
symmetric matrices $K_{1},K_{2},K_{3}$ that form a representation
of $\mathcal{S}_{0}$ and so that
\[
\left\Vert K_{r}-H_{r}\right\Vert \leq\epsilon\quad(r=1,2,3).
\]
\end{thm}

\begin{thm}
\label{thm:SphereQuaternion}
For every $\epsilon>0$ there exists a $\delta>0$ so that
whenever $H_{1},H_{2},H_{3}$ are self-dual matrices that form
a representation of $\mathcal{S}_{\delta}$, and
\[
\mathrm{Pf}\!\mathrm{-}\!\mathrm{Bott}(H_{1},H_{2},H_{3})=1,
\]
there are self-dual matrices $K_{1},K_{2},K_{3}$ that form a representation
of $\mathcal{S}_{0}$ and so that
\[
\left\Vert K_{r}-H_{r}\right\Vert \leq\epsilon\quad(r=1,2,3).
\]

\end{thm}

Theorems~\ref{thm:SphereReal} and \ref{thm:SphereQuaternion} settle
Conjectures 3 and 4 from \cite[\S VI.C]{HastingsLoringWannier}, while
Conjectures 1 and 2 from that paper were settled in our earlier paper \cite{LorSorensenDisk}.

To get theorems about unitary matrices, we utilize an old
trick from \cite{Loring-K-thryAsymCommMatrices} to turn a representation
of $\mathcal{T}_{\delta}$ into a representation
of $\mathcal{S}_{\delta}$, but modified, as in \cite{HastLorTheoryPractice},
to account for the additional symmetry. We define nonnegative real-valued
functions on the circle $f,$ $g$ and $h$ so that $f^{2}+g^{2}+h^{2}=1$
and $gh=0$ and so that
\[
(z,w)\mapsto\left(f(w),g(w)+\tfrac{1}{2}\left\{ h(w),z\right\} \right),
\]
where $\{-,-\}$ denotes the anticommutator, is a degree-one mapping of the two torus in $\mathbb{C}^{2}$ to the
unit sphere in $\mathbb{R}\times\mathbb{C}$. The exact choice only
effects the relation of $\delta$ to $\epsilon$ in the theorems below.

Given $U_{1}$ and $U_{2}$ that form a representation of $\mathcal{T}_{\delta}$
we define
\begin{align*}
H_{1} & =f(U_{2}),\\
H_{2} & =g(U_{2})
         +\tfrac{1}{4}\left\{ h(U_{2}),U_{1}^{*}\right\}
         +\tfrac{1}{4}\left\{ h(U_{2}),U_{1}\right\}, \text{ and,} \\
H_{3} & =\tfrac{i}{4}\left\{ h(U_{2}),U_{1}^{*}\right\}
         -\tfrac{i}{4}\left\{ h(U_{2}),U_{1}\right\},
\end{align*}
which then is a representation for $S_{\eta}$ where $\eta$ can be
taken small when $\delta$ is small. The anticommutator
$\{\mbox{--} ,\mbox{--} \}$ is used to ensure that
$U_r = U_r^{\mathrm{T}}$ or $U_r = U_r^{\sharp}$ propogates
to the same symmetry in the $H_r$.  Now we define
\[
\mathrm{Bott}(U_{1},U_{2})=\mathrm{Bott}(H_{1},H_{2},H_{3}),
\]
and
\[
\mathrm{Pf}\!\mathrm{-}\!\mathrm{Bott}(U_{1},U_{2})
=
\mathrm{Pf}\!\mathrm{-}\!\mathrm{Bott}(H_{1},H_{2},H_{3}).
\]

It is possible to define these invariants when the $U_r$ are
only approximately unitary.  A simple approach
is to compute the invariant as defined above but using
$\mathrm{polar}(U_r)$ in place of $U_r$.  

It is not hard to see that when these indices are nontrivial the approximation
by commuting matrices of the required form is not possible.
See \cite{HastLorTheoryPractice}.  What is not so apparent is that these indices
can be nontrivial.  In the case of almost commuting matrices that are 
unitary, and with no other restrictions, we have the example first
considered by Voiculescu \cite{voiculescu}, with $A_n$ the cylic shift 
on the basis of $\mathbb{C} ^n$ and $B_n$ a diagonal unitary.  Specifically,
when
\[
A_n = \left[\begin{array}{ccccc}
0 &  &  &  & 1\\
1 & 0\\
 & 1 & \ddots\\
 &  & \ddots & 0\\
 &  &  & 1 & 0\end{array}\right] ,
\quad
B_n = \left[\begin{array}{ccccc}
e^{2\pi i/n}\\
 & e^{4\pi i/n}\\
 &  & \ddots\\
 &  &  & e^{-2\pi i/n}\\
 &  &  &  & 1\end{array}\right],
\]
we have 
\[
\mathrm{Bott}(A_n,B_n)
= 1,
\]
while $ \| [A_n, B_n] \| \rightarrow 0$.
For a self-dual example, we pair this example with its transpose.
Using Theorem~2.8 of \cite{HastLorTheoryPractice} and the fact
the transpose commutes with the functional calculus, one can show
\[
\mathrm{Pf}\!\mathrm{-}\!\mathrm{Bott}\left(\left[\begin{array}{cc}
A_{n}\\
 & A_{n}^{\mathrm{T}}\end{array}\right],\left[\begin{array}{cc}
B_{n}\\
 & B_{n}^{\mathrm{T}}\end{array}\right]\right)= -1,
\]
while the Bott index here is trivial.  This is an example of
self-dual, almost commuting unitaries
that are close to commuting unitaries, but far from
commuting self-dual unitaries.

Here then are our other two main theorems, along with the complex
version.
\begin{thm}
\label{thm:torusComplex}
{\rm (}\cite[Corollary 6.15]{ELP-pushBusby}{\rm )}
For every $\epsilon>0$ there exists a $\delta>0$ so that whenever
matrices $U_{1},U_{2}$ form a representation of $\mathcal{T}_{\delta}$,
and
\[
\mathrm{Bott}(U_{1},U_{2})=0,
\]
there are matrices $ $$V_{1},V_{2}$ that form a representation of
$\mathcal{T}_{0}$ and so that
\[
\left\Vert U_{r}-V_{r}\right\Vert \leq\epsilon\quad(r=1,2).
\]
\end{thm}

\begin{thm}
\label{thm:torusReal}
For every $\epsilon>0$ there exists a $\delta>0$
so that whenever $U_{1},U_{2}$ are complex symmetric matrices that
form a representation of of $\mathcal{T}_{\delta}$, there are complex
symmetric matrices $V_{1},V_{2}$ that form a representation of $\mathcal{T}_{0}$
and so that
\[
\left\Vert U_{r}-V_{r}\right\Vert \leq\epsilon\quad(r=1,2).
\]
\end{thm}

\begin{thm}
\label{thm:torusQuaternion}
For every $\epsilon>0$ there exists a $\delta>0$ so that
whenever $U_{1},U_{2}$ are self-dual matrix
representations of $\mathcal{T}_{\delta}$, and
\[
\mathrm{Pf}\!\mathrm{-}\!\mathrm{Bott}(U_{1},U_{2})=1,
\]
there are self-dual matrices $V_{1},V_{2}$ that are representation
of $\mathcal{T}_{0}$ and so that
\[
\left\Vert U_{r}-V_{r}\right\Vert \leq\epsilon\quad(r=1,2).
\]
\end{thm}

\section{Block symmetries in Unstable $K$-theory}

Recall $\mathrm{Bott}(H_{1},H_{2},H_{3})$ lives in a matrix algebra
twice as big as the $H_{r}$, and observe that it can be written as
\[
\mathrm{Bott}(H_{1},H_{2},H_{3})=\sum H_{r}\otimes\sigma_{r},
\]
where 
\[
	\sigma_{1}= \begin{bmatrix}0 & 1 \\ 1 & 0\end{bmatrix}, \quad
	\sigma_{2}= \begin{bmatrix}0 & i \\ -i & 0\end{bmatrix}, \quad \text{ and } \quad
	\sigma_{3}=\begin{bmatrix} 1 & 0 \\ 0 & -1\end{bmatrix}.
\]
The $\sigma_i$, which up to sign and scaling are
the Pauli spin matrices, are anti-self-dual. When the $H_{r}$ are
complex symmetric, $\mathrm{Bott}(H_{1},H_{2},H_{3})$ is anti-self-dual.
When the $H_{r}$ are self-dual, $\mathrm{Bott}(H_{1},H_{2},H_{3})$
can be conjugated to complex anti-symmetric.
We need to be working in the original formation and so deal with the operation $\sharp\otimes\sharp$.

Recall that there are, up to isomorphism, just two $\tau$-structures that can be put on $\mathbf{M}_{n}(\mathbb{C})$, the transpose or, when $n$ is even, the dual, see \cite[\S 10.1]{li2003real}.  
We need a result about how symmetries on a matrix in $\mathbf{M}_{n}(\mathbb{C})\otimes\mathbf{M}_{2}(\mathbb{C})$ can force symmetries on one of its blocks.
This is Lemma~\ref{lem:SymmetryInPolars}.

\begin{lem}
\label{lem:bothInvertableIsDense} Suppose $\left(\mathbf{M}_{n}(\mathbb{C}),\tau\right)$
is a $C^{*,\tau}$-algebra, and consider the larger $C^{*,\tau}$-algebra
\[
\left(\mathbf{M}_{2n}(\mathbb{C}),\tau\otimes\sharp\right),
\]
meaning that on $\mathbf{M}_{2n}(\mathbb{C})$ we are using the
$\tau$-operation
\[
\left[\begin{array}{cc}
A & B\\
C & D\end{array}\right]^{\tau\otimes\sharp}=\left[\begin{array}{cc}
D^{\tau} & -B^{\tau}\\
-C^{\tau} & A^{\tau}\end{array}\right].
\]
Within the group
\[
\left\{
W\in\mathbf{M}_{2n}(\mathbb{C})
\left|\strut \ 
\det(W)=1,\ W^{*}=W^{-1}=W^{\tau\otimes\sharp}
\right.
\right\},
\]
the set of
\[
W=\left[\begin{array}{cc}
A & B\\
-B^{*\tau} & A^{\tau*}\end{array}\right],
\]
for which both $A$ and $B$ are invertible, is a dense open subset.
\end{lem}

\begin{proof}
Up to isomorphism of $C^{*,\tau}$-algebras, there are two cases,
$\tau=\mathrm{T}$ and $\tau=\sharp$. In both cases, the openness
is clear.

Suppose $\tau=\mathrm{T}$, which means $\tau\otimes\sharp=\sharp$.
First we not that the invertiblity of $A$ and $B$ holds in one case,
specifically
\begin{equation}
W_{0}=\frac{1}{\sqrt{2}}\left[\begin{array}{cc}
I & I\\
-I & I\end{array}\right],
\label{eq:ExampleOfW}
\end{equation}
which is a symplectic unitary, i.e.\  $W^{\sharp}=W^{*}=W^{-1}$.

All symplectic matrices have determinant $1$, so we can ignore the determinant
condition. We know from Lie theory that every symplectic unitary is
$e^{H}$ for a matrix $H$ with $H^{*}=H^{\sharp}=-H$. Given any
symplectic unitary
\[
W_{1}=\left[\begin{array}{cc}
A & B\\
-\overline{B} & \overline{A}\end{array}\right],
\]
we can find a matrix $H$ with $H^* = H^\sharp = -H$ such that $W_{1}=e^{H}W_{0}$.
Thus we have an analytic path $W_{t}=e^{tH}W_{0}$ from $W_{0}$ to $W_{1}$, with $W_{t}$
a symplectic unitary at every $t$. Let the upper blocks of
$W_{t}$ be $A_{t}$ and $B_{t}$. Considering the power series
for $e^{tH}$ we find convergent power series for the scalar
paths $\det(A_{t})$ and $\det(B_{t})$. As these analytic
paths are nonzero around $t=1$, neither can vanish on any open
interval. We can choose $t$ close to $0$ to find $W_{t}$ with $A_{t}$
and $B_{t}$ both invertible.

We prove the $\tau=\mathrm{\sharp}$ case similarly, starting with
the same example, $W_{1}$ as in (\ref{eq:ExampleOfW}), but now with
$n=2N$. A real orthogonal matrix of determinant one is $e^{H}$ for
a matrix with $H^{*}=H^{\mathrm{T}}=-H$. Translating this fact via
the isomorphism $\Phi$ from $\left(\mathbf{M}_{4N}(\mathbb{C}),\sharp\otimes\sharp\right)$ to $\left(\mathbf{M}_{4N}(\mathbb{C}),\mathrm{T}\right)$, which is just conjugation by a unitary (see \cite[Lemma 1.3]{HastLorTheoryPractice} or equivalently (\ref{eq:Phi-defined})), we see that when $W$ is a unitary with determinant one and $W^{*}=W^{\sharp\otimes\sharp}$, there is a matrix $H$ with $H^{*}=H^{\sharp\otimes\sharp}=-H$
so that $e^{H}=W$. 
Therefore if we are given $W_{1}$ a unitary with $W_{1}^{*}=W_{1}^{\sharp\otimes\sharp}$ we can find $H$ with $H^* = H^{\sharp \otimes \sharp} = -H$ such that $W_{1}=e^{H}W_{0}$, and so we have an analytic path $W_{t}=e^{tH}W_{0}$
of unitaries, and one can check directly, or using $\Phi$, that
$W_{t}^{*}=W_{t}^{\sharp\otimes\sharp}$.
The rest of the argument is exactly as in the previous case.
\end{proof}

\begin{lem}
\label{lem:ProductOfPolars}
If $a$ and $b$ are invertible elements of a unital
$C^{*}$-algebra and $aa^{*}+bb^{*}=1$ then
\[
\mathrm{polar}(a^{*}b)=\left(\mathrm{polar}(a)\right)^{*}\mathrm{polar}(b).
\]
\end{lem}

\begin{proof}
We use the formula $\mathrm{polar}(x)=x\left(x^{*}x\right)^{-\frac{1}{2}}$
and compute
\begin{align*}
a^{*}b\left(b^{*}aa^{*}b\right)^{-\frac{1}{2}}
 & =a^{*}b\left(b^{*}\left(1-bb^{*}\right)b\right)^{-\frac{1}{2}}\\
 & =a^{*}\left(1-bb^{*}\right)^{-\frac{1}{2}}\left(bb^{*}\right)^{-\frac{1}{2}}b\\
 & =\left(a^{*}a\right)^{-\frac{1}{2}}a^{*}b\left(b^{*}b\right)^{-\frac{1}{2}}. \qedhere
\end{align*}

\end{proof}

\begin{lem}
\label{lem:SymmetryInPolars}
If
\[
W=\left[\begin{array}{cc}
A & B\\
-B^{*\tau} & A^{\tau*}\end{array}\right],
\]
is a unitary in $\mathbf{M}_{2n}(\mathbb{C})$ with $W^{*}=W^{\tau\otimes\sharp}$ and both $A$ and $B$ are invertible, then $\left(\mathrm{polar}(A)\right)^{*}\mathrm{polar}(B)$ is fixed by $\tau$.
\end{lem}

\begin{proof}
From the fact that $W$ is unitary we deduce $AA^{*}+BB^{*}=I$ and
$A^{*}B=B^{\tau}A^{\tau*}$, so $(A^*B)^\tau = A^*B$. 
The last lemma tell us
\[
\left(\mathrm{polar}(A)\right)^{*}\mathrm{polar}(B)=\mathrm{polar}(A^{*}B).
\]
From the definition of the polar part in terms of functional calculus we find
$\left(\mathrm{polar}(x)\right)^{\tau}=\mathrm{polar}(x^{\tau})$ and therefore
\[
\left(\mathrm{polar}(A^{*}B)\right)^{\tau} = \mathrm{polar}((A^*B)^\tau) = \mathrm{polar}(A^{*}B). \qedhere
\]
\end{proof}

\section{Structure Diagonalization and $KO_{2}$}

\subsection{Hermitian anti-self-dual invertibles}

The $K$-theory we will need to track turns out to be $K_{2}$.
In the case of symmetric matrices, we will need $K_{2}$ of $\mathbf{M}_{n}(\mathbb{C})\otimes\mathbf{M}_{2}(\mathbb{C}) \cong \mathbf{M}_{2n}(\mathbb{C})$.
The fact that $K_{2}(\mathbf{M}_{2N}(\mathbb{C}),\sharp)=0$ can be
given the concrete realization that for any two Hermitian anti-self-dual
matrices that are approximately unitary, there is a symplectic unitary
that approximately conjugates one to the other.

\begin{lem}
\label{lem:DiagAntiDual-Hermitian}
Suppose $X$ in $\mathbf{M}_{2N}(\mathbb{C})$ is Hermitian and
anti-self-dual. Then there exists a symplectic unitary $W$ and
a diagonal matrix $D$ with nonnegative real diagonal entries so that
\[
X=W\left[\begin{array}{cc}
D & 0\\
0 & -D\end{array}\right]W^{*}.
\]
\end{lem}

\begin{proof}
The matrix $Y=-iX$ has $Y^{\sharp}=Y^{*}$ so we may apply results about matrices of quaternions. 
Since $Y$ is normal, indeed skew-Hermitian, the spectral theorem for matrices of quaternions \cite[Theorem 3.3]{FarenickSpectralThmquatern} gives us a symplectic unitary $W$ so that
\[
X=W\left[\begin{array}{cc}
D & 0\\
0 & -D\end{array}\right]W^{*}.
\]
This follows more easily form a different version of the quaternionic spectral theorem given in \cite[Theorem 2.4(2)]{Lor-Quaternions}, see also \cite[Page 87]{FarenickSpectralThmquatern}. 
We need only adjust $W$ and $D$ so that the diagonal of $D$ ends
up non-negative. We can swap the $j$-th diagonal element of $D$
with the $j$-th diagonal element of $-D$ as needed by utilizing
the formula
\[
\left[\begin{array}{cc}
0 & i\\
i & 0\end{array}\right]\left[\begin{array}{cc}
\lambda & 0\\
0 & -\lambda\end{array}\right]\left[\begin{array}{cc}
0 & i\\
i & 0\end{array}\right]^{*}=\left[\begin{array}{cc}
-\lambda & 0\\
0 & \lambda\end{array}\right]. \qedhere
\]
\end{proof}

\begin{thm}
\label{thm:K2_of_Quaternions}
Suppose $\left\Vert S^{2}-I\right\Vert <1$ and $S^{*}=S$ and
$S^{\sharp}=-S$. Then there is a symplectic unitary $W$ so that
\[
\left\Vert S-W\left[\begin{array}{cc}
I & 0\\
0 & -I\end{array}\right]W^{*}\right\Vert \leq\left\Vert S^{2}-I\right\Vert .
\]
\end{thm}

\begin{proof}
Lemma~\ref{lem:DiagAntiDual-Hermitian} tells us
\[
S=W^{*}\left[\begin{array}{cc}
D & 0\\
0 & -D\end{array}\right]W,
\]
for some symplectic unitary $W$ and some diagonal matrix $D$ with nonnegative real entries. 
Let $\delta=\left\Vert S^{2}-I\right\Vert$.
We compute
\[
\left\Vert D^{2}-I\right\Vert =\left\Vert \left[\begin{array}{cc}
D & 0\\
0 & -D\end{array}\right]^{2}-\left[\begin{array}{cc}
I & 0\\
0 & I\end{array}\right]\right\Vert =\left\Vert S^{2}-I\right\Vert =\delta,
\]
so the diagonal elements of $D$ are in 
\[
\left\{
\lambda\left|\strut1-\delta\leq\lambda^{2}\leq1+\delta\right.
\right\}
\subseteq
\left\{
\lambda\left|\strut1-\delta\leq\lambda\leq1+\delta\right.
\right\}.
\]
The result now follows.
\end{proof}

\subsection{Coupling two dual operations}

In the self-dual case, the $K$-theory is calculated in
$\mathbf{M}_{2N}(\mathbb{C})\otimes\mathbf{M}_{2}(\mathbb{C})\cong\mathbf{M}_{4N}(\mathbb{C})$
and the operation that specifies the symmetry on the Bott matrix is
$\sharp\otimes\sharp$. In terms of $4n$-by-$4n$ matrices, $\sharp\otimes\sharp$
is $\mathrm{T}$ conjugated by a unitary. There are many identifications
of $\mathbf{M}_{2N}(\mathbb{C})\otimes\mathbf{M}_{2}(\mathbb{C})$
with $\mathbf{M}_{4N}(\mathbb{C})$ and the one we use operates via
\[
B\otimes\begin{bmatrix}0 & 1\\
0 & 0\end{bmatrix}\mapsto\begin{bmatrix}0 & B\\
0 & 0\end{bmatrix}.
\]
Of course $B$ itself will often be written out in terms of four $N$-by-$N$
blocks. We are using
\[
Z=Z_{N}=\begin{bmatrix}0 & I\\
-I & 0\end{bmatrix}\in\mathbf{M}_{2N}(\mathbb{C}),
\]
and the operation $\sharp\otimes\sharp$ works out as
\[
\left[\begin{array}{cc}
A & B\\
C & D\end{array}\right]^{\sharp\otimes\sharp}=\left[\begin{array}{cc}
D^{\sharp} & -B^{\sharp}\\
-C^{\sharp} & A^{\sharp}\end{array}\right].
\]
The symmetry $X^{\sharp\otimes\sharp}=X^{*}$ means
\begin{equation}
X=\left[\begin{array}{cccc}
A_{11} & A_{12} & B_{11} & B_{12}\\
A_{21} & A_{22} & B_{21} & B_{22}\\
-\overline{B_{22}} & \overline{B_{21}} & \overline{A_{22}} & -\overline{A_{21}}\\
\overline{B_{21}} & -\overline{B_{11}} & -\overline{A_{21}} & \overline{A_{11}}
\end{array}\right].
\label{eq:Twisted Real}
\end{equation}

The isomorphism we need, cf. \cite[Lemma 1.3]{HastLorTheoryPractice}, is
\begin{equation}
\Phi \colon \left(\mathbf{M}_{4N}(\mathbb{C}),\sharp\otimes\sharp\right)
\rightarrow
\left(\mathbf{M}_{4N}(\mathbb{C}),\mathrm{T}\right)
\label{eq:Phi-defined}
\end{equation}
defined by
\[
\Phi\left(X\right)=UXU^{*},
\]
where $X$ is $4N$-by-$4N$ and
\[
U=\frac{1}{\sqrt{2}}\left(I\otimes I-iZ_{N}\otimes Z_{N}\right).
\]
Any $X$ in the form (\ref{eq:Twisted Real}) is conjugate to a matrix
with real entries, and so $\det(X)\in\mathbb{R}.$ This is the trick
that allows us to get our $\mathbb{Z}_{2}$-invariant as the sign
of a determinant of some invertible $X$ that is full of complex numbers.

\subsection{Hermitian anti-symmetric invertibles}

Because of the coupling of the dual operations, and the doubling of the matrix sizes in the proofs of Theorems \ref{thm:extendSphereReal} and  \ref{thm:extendSphereSelfDual}, we will be interested in the $K_{2}$-group of $\left(\mathbf{M}_{4n}(\mathbb{C}),\mathrm{T}\right)$. 
The groups is $\mathbb{Z}_{2}$, and this fact can be given the following concrete realization.
The hermitian anti-symmetric matrices fall into two classes:
those that can be approximately conjugated by an element of
$SL_{4n}(\mathbb{R})$ to
\[
S_{0}=\left[\begin{array}{cccccccccc}
 0 & (-1)^n i &&&&&& \\
(-1)^n (-i) & 0 &&&&&&& \\
 & & 0  & i &&&&& \\
 & & -i  & 0 &&&&& \\
 & & &  & 0  & i &&&\\
 & & &  & -i & 0 &&& \\
 & & &  &    &  & \ddots && \\
 & & &  &    &  & & 0 & i \\
 & & &  &    & &  & -i & 0\end{array}\right]
\]
and those that cannot.
Note that $\Pf(S_0) = (-1)^n \cdot i^{2n} = 1$. 
The sign of the Pfaffian can decide which class a matrix $X$ is in.

\begin{thm}
\label{thm:K2_of_Reals}
Suppose $\left\Vert S^{2}-I\right\Vert <1$ and $S^{*}=S$ and $S^{\mathrm{T}}=-S$ for some $S$ in $\mathbf{M}_{4n}(\mathbb{C})$.
Then there is a real orthogonal $W$ of determinant one with 
\[
	\left\Vert S-WS_{0}W^{*}\right\Vert \leq\left\Vert S^{2}-I\right\Vert 
\]
if and only if $\mathrm{Pf}(S)$ is strictly positive.
\end{thm}

\begin{proof}
We first prove the ``if'' part. 
Let $X = -iS$ and observe that $X^* = X^{\mathrm{T}} = -X$. 
It follows from \cite[Theorem 9.4]{HastLorTheoryPractice} that we can find a real orthogonal matrix $U$ and real numbers $a_1, a_2, \ldots, a_{2n}$ such that $X  = U D U^{\mathrm{T}}$, where 
\[
	D =	\begin{pmatrix}
			0 & a_1 & & & & \\
			-a_1 & 0 & & & & \\
			& & 0 & a_2 & & & \\
			& & -a_2 & 0 & & & \\
			& & & & \ddots & & \\
			& & & & & 0 & a_{2n} \\
			& & & & & -a_{2n} & 0
		\end{pmatrix}. 
\]
We may assume that the sign of $a_1$ is $(-1)^n$ and that $a_2, a_3, \ldots a_{2n} > 0$. 
We now compute the Pfaffian of $X$ in two ways. 
Using that $X = -iS$ we get
\[
	\Pf(X) = \Pf(-iS) = (-i)^{2n} \Pf(S) = (-1)^n \Pf(S).
\]
Using instead that $X = U D U^{\mathrm{T}}$ we get
\[
	\Pf(X) = \Pf(U D U^{\mathrm{T}}) = \det(U) \Pf(D) = \det(U) a_1 a_2 \cdots a_{2n}. 
\]
Hence $\det(U)$ must be positive, in particular it is $1$. 
Finally we see that
\[
	\Vert S - U S_0 U^* \Vert = \Vert iD - S_0 \Vert \leq \Vert (iD)^2 - I  \Vert = \Vert S^2 - I \Vert.
\] 

Suppose now there is $W$ in $SL_{4n}(\mathbb{R})$ so that
\[
\left\Vert S-WS_{0}W^{*}\right\Vert <1.
\]
Then $S_{1}=W^{*}SW$ satisfies $S_{1}^{*}=S_{1}$ and $S_{1}^{\mathrm{T}}=-S_{1}$ and $\left\Vert S_{1}-S_{0}\right\Vert <1$. 
The path $S_{t}=(1-t)S_{0}+tS_{1}$ is within $1$ of $S_{0}$ and $S_{0}$ is a unitary, so each $S_{t}$
is invertible and skew-symmetric. 
Therefore the path  of real numbers $\mathrm{Pf}(S_{t})$ cannot change sign.
Since $\mathrm{Pf}(S_{0}) = 1$ we must have $\mathrm{Pf}(S_{t}) > 0$ for all $t$. 
\end{proof}

\begin{thm}
\label{thm:K2_of_Reals(twisted)} 
Suppose $\left\Vert S^{2}-I\right\Vert <1$
and $S^{*}=S$ and $S^{\sharp\otimes\sharp}=-S$ for some $S$ in
$\mathbf{M}_{4n}(\mathbb{C})$. Then there is a unitary $W$ with
$W^{\sharp\otimes\sharp}=W^{*}$ with
\[
\left\Vert S-W\left[\begin{array}{cc}
I & 0\\
0 & -I\end{array}\right]W^{*}\right\Vert \leq\left\Vert S^{2}-I\right\Vert 
\]
if and only if the $K_{2}$ class represented by $S$ is trivial.
\end{thm}

\begin{proof}
We will derive this from Theorem~\ref{thm:K2_of_Reals} via the isomorphism
$\Phi$ from (\ref{eq:Phi-defined}). Recall this satisfies
$\Phi(X^{\sharp\otimes\sharp})=\left(\Phi(X)\right)^{\mathrm{T}}$
and, being a $*$-isomorphism, satisfies $\Phi(X^{*})=\left(\Phi(X)\right)^{*}$.
A short computation (with blocks in two different sizes) tells us
\[
\mathrm{Pf}\left(\Phi\left(\left[\begin{array}{cc}
I\\
 & -I\end{array}\right]\right)\right)=\mathrm{Pf}\left[\begin{array}{cccc}
0 & 0 & 0 & iI\\
0 & 0 & -iI & 0\\
0 & iI & 0 & 0\\
-iI & 0 & 0 & 0\end{array}\right]=1.
\]
We apply Theorem~\ref{thm:K2_of_Reals} to 
$\Phi\left(\begin{bmatrix}\begin{array}{cc}
I & 0\\
0 & -I\end{array}\end{bmatrix}\right)$ and find a real orthogonal $W_{1}$ so that
\[
\Phi\left(\begin{bmatrix}\begin{array}{cc}
I & 0\\
0 & -I\end{array}\end{bmatrix}\right)=W_{1}S_{0}W_{1}^{*}.
\]

The conditions on $S$ translates to $\left\Vert \Phi(S)^{2}-I\right\Vert <1$, $\Phi(S)^{*}=\Phi(S)$, and $\Phi(S)^{\mathrm{T}}=-\Phi(S)$. If the $K_2$-class of $S$ is trivial then $\mathrm{Pf}\left(\Phi(S)\right)$ is positive and so there is a real orthogonal
$W_{2}$ with
\[
\left\Vert \Phi\left(S\right)-W_{2}S_{0}W_{2}^{*}\right\Vert \leq\left\Vert \Phi(S)^{2}-I\right\Vert 
=
\left\Vert S^{2}-I\right\Vert .
\]
The desired unitary is then $W=\Phi^{-1}\left(W_{2}W_{1}^{*}\right)$. If on the other hand there is a unitary $W \in \mathbf{M}_{4n}(\mathbb{C})$ as in the statement of the theorem, then we see that $\Phi(W)W_1$ almost conjugates $\Phi(S)$ to $S_0$. Hence $\mathrm{Pf}\left(\Phi(S)\right)$ is positive, and so the $K_2$-class of $S$ is trivial. 
\end{proof}

\subsection{From $K_{-2}$ to $K_{2}$}

Behind our definition of the Bott index and the Pfaffian-Bott index
is a specific generator of $K_{-2}$ of $\left(C(S^{2}),\mathrm{id}\right)$.
The associated $R^{*}$-algebra here is $C(S^{2},\mathbb{R})$. We
skip over the actual defition of $K_{-2}$ and utilize a very important
part of Bott periodicity which is that there is a natural isomorphism
\cite{SchroderRealKtheory}
\[
K_{-2}\left(A\right)\cong K_{2}\left(A\otimes\mathbb{H}\right),
\]
for $R^{*}$-algebras, and which in terms of $C^{*,\tau}$-algebras is
\[
K_{-2}\left(A,\tau\right)
\cong 
K_{2}\left(A\otimes\mathbf{M}_{2}(\mathbb{C}),\tau\otimes\sharp\right).
\]
We take then $K_{2}\left(A\otimes\mathbf{M}_{2}(\mathbb{C}),\tau\otimes\sharp\right)$
as the definition of $K_{-2}\left(A,\tau\right)$. With this convention,
the generator of $K_{-2}\left(C(S^{2}),\mathrm{id}\right)$ is
the class of skew-$\tau$ invertible Hermitian elements represented
by
\[
b=\left[\begin{array}{cc}
x & y+iz\\
y-iz & -x\end{array}\right]
\]
where $x$, $y$ and $z$ are the coodinate functions if we regard
$S^{2}$ as the unit sphere in $\mathbb{R}^{3}$. A basic fact in
$K$-theory is that 
\[
K_{-2}\left(C(S^{2}),\mathrm{id}\right)\cong\mathbb{Z}.
\]
If we have an actual $*$-$\tau$-homomorphism $\varphi$ from $\left(C(S^{2}),\mathrm{id}\right)$
to some $\left(A,\tau\right)$, as in the proofs that follow, then
calculating $K_{-2}(\varphi)$ is just a matter of following $b$
over to where it lands in $\left(\mathbf{M}_{2}(A),\tau\otimes\sharp\right)$,
meaning
\[
\left[\begin{array}{cc}
\varphi(x) & \varphi(y)+i\varphi(z)\\
\varphi(y)-i\varphi(z) & -\varphi(x)\end{array}\right].
\]

\section{Spherical to cylindrical coordinates}

We consider various inclusions of commutative $C^{*,\tau}$-algebras,
all induced by surjections. The first we will encounter is
\begin{equation}
\iota_{1} \colon C(S^{2},\mathrm{id})\hookrightarrow C(S^{1}\times[-1,1],\mathrm{id})
\label{eq:cylindrical}
\end{equation}
which is unital and induced by 
\[
\left(e^{2\pi i\theta},t\right)
\mapsto
\left(\sqrt{1-t^{2}}\cos(\theta),\sqrt{1-t^{2}}\sin(\theta),t\right).
\]
We consider $C(S^{1}\times[-1,1],\mathrm{id})$ as universal for a
unitary $v$ and a positive contraction $k$ that commute. We also
need relations for the $\tau$ operation to create the identity involution
on the cylinder, so we require $v^{\tau}=v$ and $k^{\tau}=k$. In
term of generators and relations, the inclusion (\ref{eq:cylindrical})
operates via
\begin{align*}
h_{1} & \mapsto \frac{1}{2}\left(1-k^{2}\right)^{\frac{1}{2}}\left(v^{*}+v\right),\\
h_{2} & \mapsto \frac{1}{2}\left(1-k^{2}\right)^{\frac{1}{2}}\left(iv^{*}-iv\right),\\
h_{3} & \mapsto k.
\end{align*}

Generators and relations for real $C^{*}$-algebras have only been
considered implicitly, as in \cite[\S II]{GiordRealAF}. The relations
we need are rather basic, involving only real $C^{*}$-algebras that
are commutive or finite-dimensional, so we will not explore this topic
formally here. It has been investigated in \cite{AdamThesis}. 

\begin{thm}
\label{thm:extendSphereReal}
Suppose $(d_{n})$ is a sequence of natural numbers and that
\[
\varphi \colon C(S^{2},\mathrm{id})
\rightarrow
\prod\left(\mathbf{M}_{d_{n}}(\mathbb{C}),\mathrm{T}\right)\left/\bigoplus\left(\mathbf{M}_{d_{n}}(\mathbb{C}),\mathrm{T}\right)\right.,
\]
is a unital $*$-$\tau$-homomorphism. Then there exists a unital $*$-$\tau$-homo-morphism
$\psi$ so that
\[
\xymatrix{
C(S^1 \times [-1,1],\mathrm{id}) \ar@{-->}[dr] ^{\psi}
\\
C(S^{2},\mathrm{id})	\ar@{^{(}->}[u] ^{\iota_1}  \ar[r] _(0.3){\varphi}
	&	{\prod \left(\mathbf{M}_{d_{n}}(\mathbb{C}),\mathrm{T}\right) 
			\left/ \bigoplus \left(\mathbf{M}_{d_{n}}(\mathbb{C}),\mathrm{T}\right)
			\right.,
		}\\
}
\]
commutes.
\end{thm}

\begin{proof}
Let $\pi \colon {\prod \left(\mathbf{M}_{d_{n}}(\mathbb{C}),\mathrm{T}\right)} \twoheadrightarrow {\prod \left(\mathbf{M}_{d_{n}}(\mathbb{C}),\mathrm{T}\right) 
			\left/ \bigoplus \left(\mathbf{M}_{d_{n}}(\mathbb{C}),\mathrm{T}\right)
			\right.,
}$ be the quotient map, and let $H_{r}$ be the image under $\varphi$ of the generators $h_{r}$ of $C(S^{2})$. 
Select lifts under $\pi$ of the $H_r$, so matrices $H_{n,r}\in\mathbf{M}_{d_{n}}(\mathbb{C})$.
Taking averages and fiddling with functional calculus, we can assume
that these are contractions with
\begin{equation}
H_{n,r}=H_{n,r}^{*}=H_{n,r}^{\mathrm{T}}.
\label{eq:symmetriesOnHr}
\end{equation}
Let $\pi$ denote the quotient map. Since 
\[
H_{r}=\pi\left(\left\langle H_{1,r},H_{2,r},\dots\right\rangle \right)
\]
we have that as $n \to \infty$
\begin{equation}
\left\Vert H_{n,r}H_{n,s}-H_{n,s}H_{n,r}\right\Vert \rightarrow0\quad\mbox{ and}\quad\left\Vert \sum_r H_{n,r}^{2}-I\right\Vert 
\rightarrow 0
\label{eq:softSphere_1}
\end{equation}
When defining $\psi$ by where to send generators we can ignore any
initial segment. Therefore we may assume, without loss of generality, that
\begin{equation}
\left\Vert H_{n,r}H_{n,s}-H_{n,s}H_{n,r}\right\Vert <\frac{1}{4}
\quad\mbox{ and}\quad
\left\Vert \sum_r H_{n,r}^{2}-I\right\Vert <\frac{1}{4},
\label{eq:SmallCommutators}
\end{equation}
for all $n$. 
We apply \cite[Lemma 3.2]{HastingsLoringWannier} to
\[
S_{n}=\left[\begin{array}{cc}
H_{n,3} & H_{n,1}+iH_{n,2}\\
H_{n,1}-iH_{n,2} & -H_{n,3}\end{array}\right]
\]
to find $\left\Vert S_{n}^{2}-I\right\Vert <1$ for all $n$, and
$\left\Vert S_{n}^{2}-I\right\Vert \rightarrow0$. By (\ref{eq:symmetriesOnHr})
we have $S_{n}^{*}=S_{n}$ and $S_{n}^{\sharp}=-S_{n}$. 

Let $\delta_{n}$ be some numbers with $\delta_{n}\rightarrow0$ and
\[
\left\Vert S_n^{2}-I\right\Vert <\delta_{n}<1.
\]
Theorem~\ref{thm:K2_of_Quaternions} provides us with $A_{n}$ and
$B_{n}$ so that\[
W_{n}=\left[\begin{array}{cc}
A_{n} & B_{n}\\
-\overline{B_{n}} & \overline{A_{n}}\end{array}\right]
\]
is a symplectic unitary (so of determinant one) with
\[
\left\Vert S_{n}-W_{n}^{*}\left[\begin{array}{cc}
I & 0\\
0 & -I\end{array}\right]W_{n}\right\Vert <\delta_{n}.
\]
Lemma~\ref{lem:bothInvertableIsDense} allows us to perturb $A_{n}$
and $B_{n}$ a little so as to keep these conditions and have $A_{n}$
and $B_{n}$ invertible. Lemmas~\ref{lem:ProductOfPolars} and \ref{lem:SymmetryInPolars}
tells us
\[
\mathrm{polar}(A_{n}^{*}B_{n})=\left(\mathrm{polar}(A_{n})\right)^{*}\mathrm{polar}(B_{n}),
\]
and 
\[
\left(\mathrm{polar}(A_{n}^{*}B_{n})\right)^{\mathrm{T}}=\mathrm{polar}(A_{n}^{*}B_{n}).
\]

Now we follow the procedure in \cite[Section III.C]{HastingsLoringWannier}, which
has a history going back to \cite[Section 2]{LinFromC(X)}. However, we work
in terms of 
\[
A=\pi\left(\left\langle A_{1},A_{2},\dots\right\rangle \right),
\]
\[
B=\pi\left(\left\langle B_{1},B_{2},\dots\right\rangle \right),
\]
\[
S=\left[\begin{array}{cc}
H_{3} & H_{1}+iH_{2}\\
H_{1}-iH_{2} & -H_{3}\end{array}\right],
\]
and
\[
W=\left[\begin{array}{cc}
A & B\\
-B^{*\tau} & A^{*\tau}\end{array}\right],
\]
as we are not aiming for a quantitative result. We use $\tau$ to
denote the operation in the quotient induced by all the matrix
transpose operations. We can see that $W$ is a unitary with
\begin{equation}
S=W^{*}\left[\begin{array}{cc}
1 & 0\\
0 & -1\end{array}\right]W.
\label{eq:WconguatesS}
\end{equation}
From $WW^{*}=I$ we derive $AA^{*}+BB^{*}=1$ and $AB^{\tau}=BA^{\tau}.$
From (\ref{eq:WconguatesS}) we derive
\[
\left[\begin{array}{cc}
A^{*} & 0\\
B^{*} & 0\end{array}\right]\left[\begin{array}{cc}
A & B\\
0 & 0\end{array}\right]=\left[\begin{array}{cc}
\frac{1}{2}+\frac{1}{2}H_{3} & \frac{1}{2}H_{1}+\frac{i}{2}H_{2}\\
\frac{1}{2}H_{1}-\frac{i}{2}H_{2} & \frac{1}{2}-\frac{1}{2}H_{3}\end{array}\right],
\]
so we find
\[
A^{*}A=\frac{1}{2}+\frac{1}{2}H_{3},
\]
\[
A^{*}B=\frac{1}{2}H_{1}+\frac{i}{2}H_{2}, \text{ and }
\]
\[
B^{*}B=\frac{1}{2}-\frac{1}{2}H_{3}.
\]
Combining two of these we find $A^{*}A+B^{*}B=I.$ 

Let 
\[
Z=\pi\left(\left\langle \mathrm{polar}(A_{1}),\mathrm{polar}(A_{2}),\dots\right\rangle \right),
\]
and
\[
V=\pi\left(\left\langle \mathrm{polar}(B_{1}),\mathrm{polar}(B_{2}),\dots\right\rangle \right),
\]
so that
\[
A=Z\left(A^{*}A\right)^{\frac{1}{2}}=\left(AA^{*}\right)^{\frac{1}{2}}Z,
\]
and 
\[
B=V\left(B^{*}B\right)^{\frac{1}{2}}=\left(BB^{*}\right)^{\frac{1}{2}}V.
\]
Then
\[
V\left(\frac{1}{2}+\frac{1}{2}H_{3}\right)V^{*}=VA^{*}AV^{*}=1-VB^{*}BV^{*}=1-BB^{*}=AA^{*}
\]
and
\[
Z^{*}AA^{*}Z=A^{*}A=\frac{1}{2}+\frac{1}{2}H_{3}.
\]
Put together, these tell us
\[
V^{*}Z\left(\frac{1}{2}+\frac{1}{2}H_{3}\right)Z^{*}V=\frac{1}{2}+\frac{1}{2}H_{3},
\]
and so the unitary $U=Z^{*}V$ commutes with $H_{3}$. Since
\[
U=\pi\left(\left\langle \left(\mathrm{polar}(A_{1})\right)^{*}\mathrm{polar}(B_{1}),\left(\mathrm{polar}(A_{2})\right)^{*}\mathrm{polar}(B_{3}),\dots\right\rangle \right)
\]
we have the additional symmetry $U^{\tau}=U.$ 
Therefore we can define 
\[
	\psi \colon C(S^1 \times [-1,1],\rm{id}) \to \prod\left(\mathbf{M}_{d_{n}}(\mathbb{C}),\mathrm{T}\right)\left/\bigoplus\left(\mathbf{M}_{d_{n}}(\mathbb{C}),\mathrm{T}\right)\right.,
\]
by sending the universal generators of $C(S^1 \times [-1,1],\rm{id})$ to $U$ and $H_{3}$. 
This makes the diagram
commute because
\begin{align*}
U\left(1-H_{3}^{2}\right)^{\frac{1}{2}} & =2Z^{*}V\left(B^{*}B\right)^{\frac{1}{2}}\left(A^{*}A\right)^{\frac{1}{2}}\\
 & =2Z^{*}\left(1-BB^{*}\right)^{\frac{1}{2}}B\\
 & =2Z^{*}\left(AA^{*}\right)^{\frac{1}{2}}B\\
 & =2A^{*}B\\
 & =H_{1}+iH_{2}. \qedhere
\end{align*}
\end{proof}

\begin{thm}
\label{thm:extendSphereSelfDual}
Suppose $d_{n}$ is a sequence of
natural numbers and that 
\[
\varphi \colon C(S^{2},\mathrm{id})\rightarrow\prod\left(\mathbf{M}_{2d_{n}}(\mathbb{C}),\sharp\right)\left/\bigoplus\left(\mathbf{M}_{2d_{n}}(\mathbb{C}),\sharp\right)\right.
\]
is a unital $*$-$\tau$-homomorphism. If $K_{-2}(\varphi)=0$ then
there exists a unital $*$-$\tau$-homomorphism $\psi$ so that
\[
\xymatrix{
C(S^1 \times [-1,1],\mathrm{id}) \ar@{-->}[dr] ^{\psi}
\\
C(S^{2},\mathrm{id})	\ar@{^{(}->}[u] ^{\iota_1}  \ar[r] _(0.3){\varphi}
	&	{\prod \left(\mathbf{M}_{2d_{n}}(\mathbb{C}),\sharp \right) 
			\left/ \bigoplus \left(\mathbf{M}_{2d_{n}}(\mathbb{C}),\sharp\right)
			\right.
		}\\
}
\]
commutes.
\end{thm}

\begin{proof}
The proof begins as before, with $H_{n,r}$ in $\mathbf{M}_{2d_{n}}(\mathbb{C})$
and with (\ref{eq:symmetriesOnHr}) replaced by
\[
H_{n,r}=H_{n,r}^{*}=H_{n,r}^{\mathrm{\sharp}}.
\]
After making the reduction to get (\ref{eq:SmallCommutators}), the
Pfaffian-Bott index is well defined for all $n$. The assumption on
the $K$-theory of $\varphi$ tells us that this index is zero for
large $n$. Further truncating the sequences, we can reduce to the
case where $S_{n}$ represents the trivial $K_{2}$ class for all
$n$. 

Let $\delta_{n}$ be some numbers with $\delta_{n}\rightarrow0$ and
\[
\left\Vert S^{2}-I\right\Vert <\delta_{n}<1.
\]
Theorem~\ref{thm:K2_of_Reals(twisted)} provides us with 
\[
W_{n}=\left[\begin{array}{cc}
A_{n} & B_{n}\\
-\overline{B_{n}} & \overline{A_{n}}\end{array}\right]
\]
that is a unitary of determinant one with $W_{n}^{\sharp\otimes\sharp}=W_{n}^{*}$
and
\[
\left\Vert S_{n}-W_{n}^{*}\left[\begin{array}{cc}
I & 0\\
0 & -I\end{array}\right]W_{n}\right\Vert <\delta_{n}.
\]
Lemma~\ref{lem:bothInvertableIsDense} again allows us to reduce
to the case where $A_{n}$ and $B_{n}$ are invertible. Lemmas~\ref{lem:ProductOfPolars}
and \ref{lem:SymmetryInPolars} tell us
\[
\mathrm{polar}(A_{n}^{*}B_{n})=\left(\mathrm{polar}(A_{n})\right)^{*}\mathrm{polar}(B_{n}),
\]
and 
\[
\left(\mathrm{polar}(A_{n}^{*}B_{n})\right)^{\sharp}=\mathrm{polar}(A_{n}^{*}B_{n}).
\]
The rest of the proof goes though unchanged, although $\tau$ in the
quotient is that derived from the sequence of $\sharp$ operations.
\end{proof}

We now wish to proceed as in \cite[Section 6.3]{ELP-pushBusby}, and
so need a way to ``poke holes'' in open patches. Let $\Omega$ denote
the open unit disk
\[
\Omega=\left\{ (x,y)\left|\strut x^{2}+y^{2}<1\right.\right\} 
\]
and 
\[
\Omega^{[1]}=\left\{ (x,y)\left|\strut1\leq x^{2}+y^{2}<2\right.\right\} .
\]
We have a proper surjective continuous map $\Omega^{[1]}\rightarrow\Omega$
srinking the inner circle to a point. 
We think of $\Omega^{[1]}$ as $\Omega$ with $\mathbf{0}$ replaced
by a circle. See Figure~\ref{fig:The-map-of-Omega-1}.
This gives us a map 
\[
\iota \colon C_{0}\left(\Omega\right)\hookrightarrow C_{0}\left(\Omega^{[1]}\right).
\]

\begin{figure}
\includegraphics[bb=00bp 00bp 500bp 400bp,clip,scale=0.5]{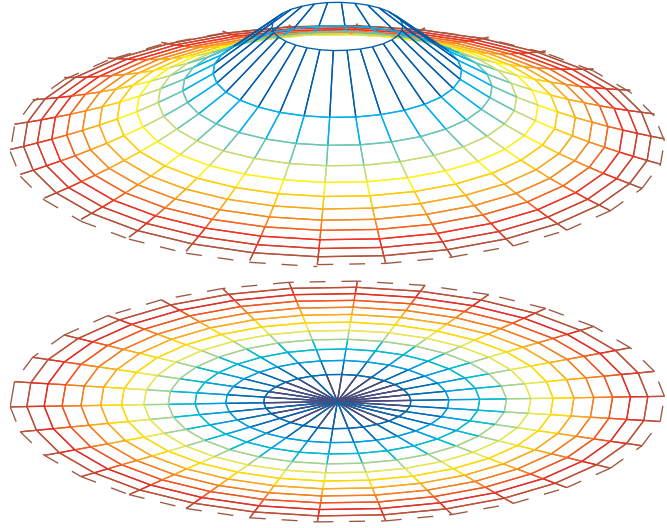}
\caption{The map of $\Omega^{[1]}$ onto $\Omega$.\label{fig:The-map-of-Omega-1}}
\end{figure}

\begin{cor}
\label{cor:OneVolcationOnOmega} 
Suppose 
\[
\varphi \colon C_{0}(\Omega,\mathrm{id})
\rightarrow
\prod\left(\mathbf{M}_{d_{n}}(\mathbb{C}),\tau_{n}\right)\left/\bigoplus\left(\mathbf{M}_{d_{n}}(\mathbb{C}),\tau_{n}\right)\right.
\]
is a $*$-$\tau$-homomorphism with $\tau_{n}=\mathrm{T}$ for all $n$ or
$\tau_{n}=\sharp$ and $d_n$ even for all $n$.
If $K_{-2}(\varphi)=0$ then
there exists a $*$-$\tau$-homomorphism $\psi$ so that 
\[
\xymatrix{
C(\Omega^{[1]},\mathrm{id}) \ar@{-->}[dr] ^{\psi}
\\
C(\Omega,\mathrm{id})	\ar@{^{(}->}[u]^-{\iota}  \ar[r]_-{\varphi}
	&	{\prod \left(\mathbf{M}_{d_{n}}(\mathbb{C}),\tau_n \right) 
			\left/ \bigoplus \left(\mathbf{M}_{d_{n}}(\mathbb{C}),\tau_n \right)
			\right.
		}\\
}
\]
commutes.
\end{cor}

\begin{proof}
The inclusion
$C_{0}(\Omega,\mathrm{id})\hookrightarrow\left(C_{0}(\Omega,\mathrm{id})\right)^{\sim}$
into the unitization is an isomorphism on $K_{-2},$ (essentially
by definition since $K_{-2}(\mathbb{R})=0$). Therefore it suffices
to prove the equivalent unital extension problem where we unitize
the two commutative $C^{*,\tau}$-algebras:
\[
\xymatrix{
C(\mathbb{D},\mathrm{id}) \ar@{-->}[dr]^-{\widetilde{\psi}}
\\
C(S^2,\mathrm{id})	\ar@{^{(}->}[u]^{\iota_2}  \ar[r]_-{\widetilde{\varphi}}
	&	{\prod \left(\mathbf{M}_{d_{n}}(\mathbb{C}),\tau_n \right) 
			\left/ \bigoplus \left(\mathbf{M}_{d_{n}}(\mathbb{C}),\tau_n \right)
			\right.
		}\\
}
\]
Here $\iota_{2}$ comes from the degree-one map $\mathbb{D}\rightarrow S^{2}$
that sends the origin to the north pool and the boundary circle to
the south pole. We have the factorization
\[
\xymatrix@C=1.2cm{
C(S^2) \ar[r]^-{\iota_2} \ar@/_0.6cm/[rr]_-{\iota_0}
	& C(\mathbb{D},\mathrm{id}) \ar[r]
		& C(S^1 \otimes [-1,1],\mathrm{id})
}
\]
and so this follows from Theorems~\ref{thm:extendSphereReal} and
\ref{thm:extendSphereSelfDual}.
\end{proof}

To prove our results for the torus geometry, we need to know about amalgamated products of
$C^{*,\tau}$-algebras.  For the sphere geometry, we are ready to give a proof.  That is,
we now prove Theorems~\ref{thm:SphereReal} and \ref{thm:SphereQuaternion}.

Given
\[
\varphi \colon C(S^2)
\rightarrow
\prod\left(\mathbf{M}_{d_{n}}(\mathbb{C}),\tau_{n}\right)\left/\bigoplus\left(\mathbf{M}_{d_{n}}(\mathbb{C}),\tau_{n}\right)\right.
\]
with trivial $K_2$ we have, as we saw in the proof of
Corollary~\ref{cor:OneVolcationOnOmega},
an extension to a $*$-$\tau$-homomorphims from $C(\mathbb{D})$.  This
we can lift to the product by the main result in \cite{LorSorensenDisk}.
This solves the lifting problem, at least when the $K$-theory allows it,
for the map from $C(S^2)$.  We leave to the reader the usual conversion of
Theorems~\ref{thm:SphereReal} and \ref{thm:SphereQuaternion} into
a lifting problem via generators and relations.

\section{Amalgamated products of $C^{*,\tau}$-algebras}

Define $\Upsilon_{A} \colon A\rightarrow A^{\mathrm{op}}$ to be the identity
map on the underlying set, so a $*$-anti-homomorphism. We have 
$\left(A^{\mathrm{op}}\right)^{\mathrm{op}}=A.$ 

\begin{lem}
Suppose $C,$ $A_{1}$ and $A_{2}$ are $C^{*}$-algebras,
$\theta_{1} \colon C\rightarrow A_{1}$ and $\theta_{2} \colon C\rightarrow A_{2}$
are  $*$-homomorphisms and consider the associated amalgamated product 
$A_{1}\ast_{C}A_{2}$ and canonical $*$-homomorphisms 
$\iota_{j} \colon A_{j}\rightarrow A_{1}\ast_{C}A_{2}.$
If $\varphi_{j} \colon A_{1}\rightarrow D$ and $\varphi_{2} \colon A_{2}\rightarrow D$
are $*$-anti-homomorphisms such that
$\varphi_{1}\circ\theta_{1}=\varphi_{2}\circ\theta_{2},$
then there is a unique $*$-anti-homomorphism $\Phi \colon A_{1}\ast_{C}A_{2}\rightarrow D$
such that $\Phi\circ\iota_{j}=\varphi_{j}.$
\end{lem}

\begin{proof}
Both $\Upsilon_{D}\circ\varphi_{1}$ and $\Upsilon_{D}\circ\varphi_{2}$
are $*$-homomorphisms, and
$\Upsilon_{D}\circ\varphi_{1}\circ\theta_{1}=\Upsilon_{D}\circ\varphi_{2}\circ\theta_{2}.$
There is a unique $*$-homomorphism
$\Psi \colon A_{1}\ast_{C}A_{2}\rightarrow D^{\mathrm{op}}$
such that $\Psi\circ\iota_{j}=\Upsilon_{D}\circ\varphi_{j}.$ Let
$\Phi=\Upsilon_{D^{\mathrm{op}}}\circ\Psi$.
This is a $*$-anti-homomorphism
and
\[
\Phi\circ\iota_{j}
=\Upsilon_{D^{\mathrm{op}}}\circ\Psi\circ\iota_{j}
=\Upsilon_{D^{\mathrm{op}}}\circ\Upsilon_{D}\circ\varphi_{j}
=\varphi_{j}.
\]
 If $\Phi^{\prime} \colon A_{1}\ast_{C}A_{2}\rightarrow D$ is also a $*$-anti-homomorphism
with $\Phi^{\prime}\circ\iota_{j}=\varphi_{j}$ then $\Upsilon_{D}\circ\Phi^{\prime}$
is a $*$-homomorphism and 
\[
\Upsilon_{D}\circ\Phi^{\prime}\circ\iota_{j}=\Upsilon_{D}\circ\varphi_{j},
\]
therefore $\Upsilon_{D}\circ\Phi^{\prime}=\Psi$ and so
\[
\Phi^{\prime}
=\mathrm{id}_{D}\circ\Phi^{\prime}
=\Upsilon_{D^{\mathrm{op}}}\circ\Upsilon_{D}\circ\Phi^{\prime}
=\Upsilon_{D^{\mathrm{op}}}\circ\Psi
=\Phi. \qedhere
\]
\end{proof}

What the lemma shows is that not only is $A_{1} \ast_{C} A_{2}$ universal for $*$-homomorphisms out of $A_1$ and $A_2$, it is also universal for anti-$*$-homomorphisms. 
We will use this fact to put a $\tau$-structure on the amalgamated product. 

\begin{thm}
\label{thm:push-outWithTau}
Suppose $\left(C,\tau_{0}\right),$ $\left(A_{1},\tau_{1}\right)$
and $\left(A_{1},\tau_{2}\right)$ are $C^{*,\tau}$-algebras and
$\theta_{1} \colon C\rightarrow A_{1}$ and $\theta_{2} \colon C\rightarrow A_{2}$
are $*$-$\tau$-homomorphisms. The $C^{*,\tau}$-algebra $A_{1}\ast_{C}A_{2}$
becomes a $C^{*,\tau}$-algebra with the unique operation $\tau$
making both $\iota_{j} \colon A_{j}\rightarrow A_{1}\ast_{C}A_{2}$ into
$*$-$\tau$-homomorphisms. Moreover, $\left(A_{1}\ast_{C}A_{2},\tau\right)$
is the amalgamated product of $\left(A_{1},\tau_{1}\right)$ and $\left(A_{1},\tau_{2}\right)$
over $\left(C,\tau_{0}\right)$.
\end{thm}

\begin{proof}
We have anti-homomorphisms $T_{0} \colon C\rightarrow C$ and
$T_{j} \colon A_{j}\rightarrow A_{j}$ defined by
$T_{0}(c)=c^{\tau_{0}}$ and $T_{j}(a)=a^{\tau_{j}}$.
The statement that $\theta_{j}$ is a $*$-$\tau$-homomorphism translates
to $\theta_{j}\circ T_{0}=T_{j}\circ\theta_{j}$. 
So we have 
\[
	\iota_1 \circ T_1 \circ \theta_1 = \iota_1 \circ \theta_1 \circ T_0 = \iota_2 \circ \theta_2 \circ T_0 = \iota_2 \circ T_2 \circ \theta_2.
\]
Therefore there is a unique $*$-anti-homomorphism
$T \colon A_{1}\ast_{C}A_{2}\rightarrow A_{1}\ast_{C}A_{2}$
such that $T(\iota_{j}(a))=\iota_{j}(a^{\tau_{j}})$. Certainly $T\circ T$
is a $*$-homomorphisms, and since it fixes $\iota_{j}(a)$ it is
the identity map. We thus can make $A_{1}\ast_{C}A_{2}$ into a
$C^{*,\tau}$-algebra by $x^{\tau}=T(x)$. For example,
\[
\left(\iota_{1}(a)\iota_{2}(b)\iota_{1}(c)\right)^{\tau}
=\iota_{1}(c^{\tau_{1}})\iota_{2}(b^{\tau_{2}})\iota_{1}(a^{\tau_{1}}).
\]

If we are given $\varphi_{j} \colon A_{j}\rightarrow B$, two $*$-$\tau$-homomorphisms
such that $\varphi_{1}\circ\theta_{1}=\varphi_{2}\circ\theta_{2}$,
then there is a unique $*$-homomorphism $\Phi \colon A_{1}\ast_{C}A_{2}\rightarrow B$
such that $\Phi(\iota_{j}(a))=\varphi_{j}(a)$ for all $a$ in $A_{j}$.
To finish, we must show that $\Phi$ is actually a $*$-$\tau$-homomorphism.
Given a product 
\[
w=\iota_{1}(a_{1})\iota_{2}(b_{1})\iota_{1}(a_{2})\iota_{2}(b_{2})\cdots\iota_{1}(a_{n})\iota_{2}(b_{n})\]
we find
\begin{align*}
\Phi(w^{\tau}) & =\Phi\left(\iota_{2}(b_{n}^{\tau_{2}})\iota_{1}(a_{n}^{\tau_{1}})\cdots\iota_{2}(b_{1}^{\tau_{2}})\iota_{1}(a_{1}^{\tau_{1}})\right)\\
 & =\varphi_{2}(b_{n}^{\tau_{2}})\varphi_{1}(a_{n}^{\tau_{1}})\cdots\varphi_{2}(b_{1}^{\tau_{2}})\varphi_{1}(a_{1}^{\tau_{1}})\\
 & =\varphi_{2}(b_{n})^{\tau}\varphi_{1}(a_{n})^{\tau}\cdots\varphi_{2}(b_{1})^{\tau}\varphi_{1}(a_{1}^{\tau}))\\
 & =\left(\varphi_{1}(a_{1})\varphi_{2}(b_{1})\cdots\varphi_{1}(a_{n})\varphi_{2}(b_{n})\right)^{\tau}\\
 & =\left(\Phi\left(\iota_{1}(a_{1})\iota_{2}(b_{1})\cdots\iota_{1}(a_{n})\iota_{2}(b_{n})\right)\right)^{\tau}\\
 & =\left(\Phi\left(w\right)\right)^{\tau}.
\end{align*}
For the other three types of words (ending in $\iota_{1}(A_{1})$
or beginning in $\iota_{2}(A_{2})$) we also find $\Phi(w^{\tau})=\Phi(w)^{\tau}$.
Since $\Phi$ is linear and continuous, we conclude $\Phi(x^{\tau})=\Phi(x)^{\tau}$
for all $x$ in $A_{1}\ast_{C}A_{2}$.
\end{proof}

We can talk about push-out diagrams instead of amalgamated products.
If $C$ is large relative to $A_{1}$ or $A_{2}$ this tends to be
a more fitting description. What Theorem~\ref{thm:push-outWithTau}
shows is that when a diagram
\[
\xymatrix@R=0.5cm@C=0.45cm{
	& D
\\
A	\ar[ru]
	&	& B, \ar[lu]
\\
	&C \ar[ru] \ar[lu]
}
\]
is a diagram
of $*$-$\tau$-homomorphisms and $C^{*,\tau}$-algebras, if it is
a push-out in the category of $C^{*}$-algebras then it is a push-out
in the category of $C^{*,\tau}$-algebras.

\begin{defn}
We say that a $*$-$\tau$-homomorphism is \emph{proper} if it is
proper as a $*$-homomorphism. 
\end{defn}

Recall (or see \cite{ELP-pushBusby}) that a $*$-homomorphism $\varphi \colon A\rightarrow B$
is said to be proper if for some (equivalently every) approximate
unit $e_{\lambda}$ of $A$ the image $\varphi(e_{\lambda})$ is an
approximate unit for $B$. Using Cohen factorization we find that this is equivalent to the condition $B=\varphi(A)B$, see \cite{Pedersen-Factorization}.  

\begin{thm} \label{thm:properGivespush-out}
If a diagram of extensions
\[
\xymatrix{
0 \ar[r]
	& A_1 \ar[r] 
		& C_1 \ar[r] 
			& B \ar[r]
				& 0 \\
0 \ar[r]
	& A \ar[r] \ar[u] ^{\alpha}
		& C \ar[r] \ar[u]
			& B \ar[r] \ar@{=}[u]
				& 0 \\
}
\]
of $C^{*,\tau}$-algebras
is given, with $\alpha$ proper, then the left square is a push-out. 
\end{thm}

\begin{proof}
This is an immediate corollary of this statement in the category
of $C^{*}$-algebras \cite[Corollary 4.3]{ELP-pushBusby}
and Theorem~\ref{thm:push-outWithTau}.
\end{proof}

\section{The sphere, torus and other 2D situations}

Suppose $X$ is a two-dimensional, finite CW complex. Assume the $2$-cells
are given as unit squares, and on that unit square a lattice of
size $\frac{1}{m}$, at the center point of every square created by
the lattice grid we insert a circle and so create a space $X^{[m]}$.
Let $\Gamma_{m}$ denote the closed subset of $X$ corresponding to
replacing each $2$-cell with the 1-D mesh. Consult figure \ref{fig:XandX[n]}.

\begin{figure}
\includegraphics[bb=00bp 00bp 455bp 160bp,clip,scale=0.8]{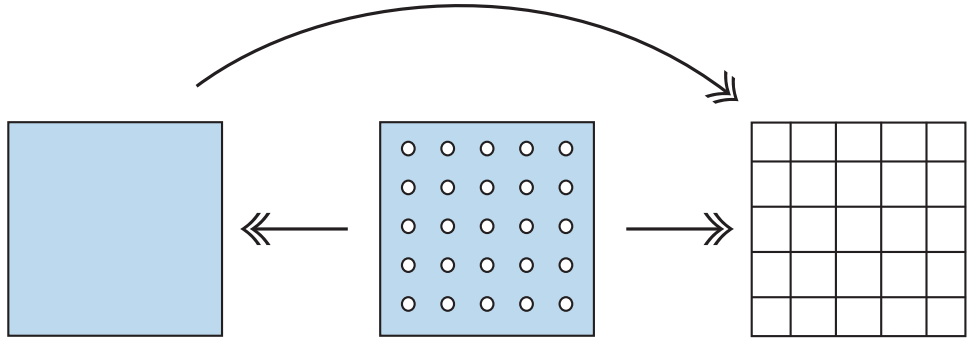}
\caption{How we treate each 2-cell in a 2-dimensional CW complex $X$.
\label{fig:XandX[n]} }
\end{figure}

\begin{thm} \label{thm:extendAnyCW}
Suppose $X$ is a finite two-dimensional CW complex and that
\[
\varphi \colon C(X,\mathrm{id})
\rightarrow
\prod\left(\mathbf{M}_{d_{n}}(\mathbb{C}),\tau_{n}\right)\left/\bigoplus\left(\mathbf{M}_{d_{n}}(\mathbb{C}),\tau_{n}\right)\right.
\]
is a unital $*$-$\tau$-homomorphism with $\tau_{n}=\mathrm{T}$ for all $n$ or $\tau_{n}=\sharp$ and $d_n$ even for all $n$.
If $K_{-2}(\varphi)=0$ then there exists a $*$-$\tau$-homomorphism $\psi$ so that
\[
\xymatrix{
C(X^{[n]},\mathrm{id}) \ar@{-->}[dr]^-{\psi}
\\
C(X,\mathrm{id})	\ar@{^{(}->}[u]^{\iota}  \ar[r]_-{\varphi}
	&	{\prod \left(\mathbf{M}_{d_{n}}(\mathbb{C}),\tau_n \right) 
			\left/ \bigoplus \left(\mathbf{M}_{d_{n}}(\mathbb{C}),\tau_n \right),
			\right.
		}\\
}
\]
commutes.
\end{thm}

\begin{proof}
Suppose $Y$ is an open set in $X$ and that $Y$ is homeomorphic to the open unit disk $\Omega$. 
Let $Z$ be the space we get from $X$ if we remove the closed set in $Y$ that corresponds to the open disk of radius $\tfrac{1}{2}$ at the center of $\Omega$. 
Denote by $W$ the open disc with a hole in $Z$, with notational abuse $W = Z \cap Y$.
Map $Z$ onto $X$ by a continuous function that fixes points in $X \setminus W$ and inside $W$ it corresponds to the surjection 
\[
(W = )\left\{ (x,y)\left|\strut\tfrac{1}{2}\leq x^{2}+y^{2}< 1 \right.\right\} 
\twoheadrightarrow
\left\{ (x,y)\left|\strut x^{2}+y^{2}<1\right.\right\} 
(=\Omega),
\]
that shrinks the hole in middle to a point. 
Let $\iota \colon C(X,\mathrm{id}) \hookrightarrow C(Z, \mathrm{id})$ denote the induced inclusion. 

We can complete the proof by induction, if we can show the following extension property 
\[
\xymatrix{
C(Z,\mathrm{id}) \ar@{-->}[dr]^-{\psi}
\\
C(X,\mathrm{id})	\ar@{^{(}->}[u]^{\iota} \ar[r]_-{\varphi}
	&	{\prod \left(\mathbf{M}_{d_{n}}(\mathbb{C}),\tau_n \right) 
			\left/ \bigoplus \left(\mathbf{M}_{d_{n}}(\mathbb{C}),\tau_n \right)
			\right.,
		}\\
}
\]
where $\varphi$ is assumed unital and $\psi$ is required to be unital.

By using homeomorphisms we can replace $W$ by $\Omega^{[1]}$ and $Y$ by $\Omega$. 
Therefore we get a commutative diagram of $C^{*,\tau}$-algebras 
\[
\xymatrix{
0 \ar[r]
	& C_0(\Omega^{[1]},\mathrm{id}) \ar[r] 
		& C(Z,\mathrm{id}) \ar[r] 
			& C(X\setminus Y,\mathrm{id}) \ar[r]
				& 0 \\
0 \ar[r]
	& C_0(\Omega,\mathrm{id}) \ar[r] \ar[u]^{\alpha}
		& C(X,\mathrm{id}) \ar[r] \ar[u]
			& C(X\setminus Y,\mathrm{id}) \ar[r] \ar@{=}[u]
				& 0, \\
}
\]
with exact rows. 
The map $\Omega^{[1]} \rightarrow \Omega$ inducing $\alpha$ is proper, hence $\alpha$ is proper and by Theorem~\ref{thm:properGivespush-out} the left square is a push-out.
By Corollary~\ref{cor:OneVolcationOnOmega} we can find a $*$-$\tau$-homomorphism $\lambda$ such that 
\[
\xymatrix{
	C_0(\Omega^{[1]},\mathrm{id}) \ar@/^1em/[drr]^-{\lambda} & & \\
	C_0(\Omega,\mathrm{id}) \ar[r] \ar[u]^{\alpha} & C(X,\mathrm{id}) \ar[r]_-{\phi} & {\prod \left(\mathbf{M}_{d_{n}}(\mathbb{C}),\tau_n \right) \left/ \bigoplus \left(\mathbf{M}_{d_{n}}(\mathbb{C}),\tau_n \right) \right.,}
}
\]
commutes. 
By the universal property of push-outs, $\lambda$ and $\phi$ combine to define a $*$-homomorphism 
\[
	\psi \colon C(Z,\mathrm{id}) \to {\prod \left(\mathbf{M}_{d_{n}}(\mathbb{C}),\tau_n \right) \left/ \bigoplus \left(\mathbf{M}_{d_{n}}(\mathbb{C}),\tau_n \right) \right.,}
\]
that extends $\phi$.  
As the unit in $C(Y)$ comes from the unit in $C(X)$ it is automatic that $\psi$ is unital.
\end{proof}

\begin{thm} \label{thm:C(X)isRealMSP}
Suppose $X$ is a finite two-dimensional CW complex.
Then for any unital $*$-$\tau$-homomorphism
\[
\varphi \colon C(X,\mathrm{id})\rightarrow\prod\left(\mathbf{M}_{d_{n}}(\mathbb{C}),\tau_{n}\right)\left/\bigoplus\left(\mathbf{M}_{d_{n}}(\mathbb{C}),\tau_{n}\right)\right.,
\]
with $\tau_{n}=\mathrm{T}$ for all $n$ or $\tau_{n}=\sharp$ and $d_n$ even for
all $n$.
If $K_{-2}(\phi) = 0$ then there exists a unital $*$-$\tau$-homomorphism $\psi$
making the following diagram commute: 
\[
\xymatrix{
	&	{\prod \left(\mathbf{M}_{d_{n}}(\mathbb{C}),\tau_n \right) }
	\ar[d] ^{\pi}
\\
C(X,\mathrm{id})	 \ar[r] _(0.25){\varphi} \ar@{-->}[ru] ^{\psi}
	&	{\prod \left(\mathbf{M}_{d_{n}}(\mathbb{C}),\tau_n \right) 
			\left/ \bigoplus \left(\mathbf{M}_{d_{n}}(\mathbb{C}),\tau_n \right)
			\right.}.
\\
}
\]
\end{thm}

\begin{proof}
Let $\mathcal{F}$ be a finite generating set for $C(X)$.
By an intertwining argument as in the proof of Theorem~3.1 in \cite{EilersLoringContingenciesStableRelations} we can show it suffices to find for each $\epsilon>0$ a $\psi$ so that $\left\Vert \psi(f)-\pi\circ\varphi(f)\right\Vert <\epsilon$ for all $f$ in $\mathcal{F}$.
Using Theorem \ref{thm:extendAnyCW} we can extend $\varphi$ to $C(X^{[m]},\mathrm{id})$ for large $m$ and then by retracting $X^{[m]}$ to $\Gamma_{n}$ we get that $\varphi$ approximately factors through $C(\Gamma_{n},\mathrm{id})$.
A map that factors through $C(\Gamma_{n},\mathrm{id})$ will lift exactly, by the semiprojectitity of $C(\Gamma_{n},\mathrm{id})$ which was established in our previous paper, \cite{LorSorensenDisk}.
\end{proof}

If we can describe $C(X)$ by generators and relations, we get corollaries
about approximation of tuples of real or self-dual matrices. Two spaces
for which this works well are the familiar sphere and torus. 

We have thus proven the two remaining main theorems, Theorems \ref{thm:torusReal} and \ref{thm:torusQuaternion}.

\section{\label{sec:-WannierSpread} Controlling maximum Wannier spread}

Commuting sets of Hermitian matrices can be simultaneously diagonalized
by a unitary matrix. Thus our main results have equivalent formulations
in terms of simultaneous approximate diagonalization. Simultateous
approximate diaganalization arises in signal processing \cite{cardosoSimultanDiagn}
and the resulting algorithms have been unitilized in first principals
molecular dymamics \cite{gygi2003computation}. 

There are many ways to measure the failure of a set of matrices to
be diagonal. Essentially every unitarly invariant norm of the matrices
gives a means to measure this. If the matrices can be interpreted
as position operators, one can instead ask for ``exponential localization''
which really only makes sense if we are considering a class of finite models
of increasing lattice size. We bring this up because it is probably
the most physically relevant measure. Naturally, it is hard to even
define, much less compute numerically. Within the choices of norms,
the operator norm would be the most physically relevant, while something
like the Frobenius norm would be the easiest to incorporate into an efficient
algorithm.

In condensed matter physics, one often looks for a basis for an energy
band that is well localized. When computed numercially, generally
one attempts to minimize the \emph{total} Wannier spread, which corresponds
to minimizing off-diagonal parts in the Frobenius norm. We consider
theoretical bounds on minimizing the \emph{maximum }Wannier spread.

Given $X_{1},\dots,X_{d}$ Hermitian matrices in $\mathbf{M}_{n}(\mathbb{C})$
and a unit vector $\mathbf{b}$, the \emph{Wannier spread of} $\mathbf{b}$
with respect to the set $\mathcal{X}$ of these of $X_{r}$ we define
as
\[
\sigma_{\mathcal{X}}^{2}(\mathbf{b})
=\sum_{r=1}^{d}\left\langle X_{r}^{2}\mathbf{b},\mathbf{b}\right\rangle -\left\langle X_{r}\mathbf{b},\mathbf{b}\right\rangle ^{2}.
\]
Given an orthonormal subset $\mathcal{B}=\{\mathbf{b}_{1},\dots,\mathbf{b}_{k}\}$
of $\mathbb{C}^{n}$ we define its \emph{total Wannier spread} with
respect to $\mathcal{X}$ as
\[
\sum_{j}\sigma_{\mathcal{X}}^{2}(\mathbf{b}_{j}),
\]
and its \emph{maximum Wannier spread }with resect to $\mathcal{X}$
as
\[
\mu_{\mathcal{X}}(\mathcal{B})=\max_{j}\sigma_{\mathcal{X}}^{2}(\mathbf{b}_{j}).
\]

For $\mathcal{X}=\{X_{1},\dots,X_{d}\}$ and $\mathcal{Y}=\{Y_{1},\dots,Y_{d}\}$,
both sets of Hermitian matices on $\mathbb{C}^{n}$, define
\[
\left\Vert \mathcal{X}\right\Vert =\max_{r}\left\Vert X_{r}\right\Vert,  
\]
and
\[
\mathrm{dist}\left(\mathcal{X},\mathcal{Y}\right)=\max_{r}\left\Vert X_{r}-Y_{r}\right\Vert .
\]

\begin{lem}
\label{lem:continuityOfSpread} 
Suppose $\mathcal{X}=\{X_{1},\dots,X_{d}\}$
and $\mathcal{Y}=\{Y_{1},\dots,Y_{d}\}$ are sets of Hermitian matices
on $\mathbb{C}^{n}$ with $\left\Vert \mathcal{X}\right\Vert \leq1$
and $\left\Vert \mathcal{Y}\right\Vert \leq1$. For any orthonormal
subset $\mathcal{B}=\{\mathbf{b}_{1},\dots,\mathbf{b}_{k}\}$ of $\mathbb{C}^{n}$,
\[
\mu_{\mathcal{X}}(\mathcal{B})\leq\mu_{\mathcal{Y}}(\mathcal{B})+4d \cdot \mathrm{dist}\left(\mathcal{X},\mathcal{Y}\right).
\]
\end{lem}

\begin{proof}
For any unit vector $\mathbf{b}$ we find
\[
\left|\left\langle X_{r}^{2}\mathbf{b},\mathbf{b}\right\rangle -\left\langle Y_{r}^{2}\mathbf{b},\mathbf{b}\right\rangle \right|\leq\left\Vert X_{r}^{2}-Y_{r}^{2}\right\Vert \leq2\left\Vert X_{r}-Y_{r}\right\Vert, 
\]
and 
\[
\left|\left\langle X_{r}\mathbf{b},\mathbf{b}\right\rangle ^{2}-\left\langle Y_{r}\mathbf{b},\mathbf{b}\right\rangle ^{2}\right|\leq2\left|\left\langle X_{r}\mathbf{b},\mathbf{b}\right\rangle -\left\langle Y_{r}\mathbf{b},\mathbf{b}\right\rangle \right|\leq2\left\Vert X_{r}-Y_{r}\right\Vert, 
\]
and so
\[
\left|\sigma_{X_{r}}^{2}(\mathbf{b})-\sigma_{Y_{r}}^{2}(\mathbf{b})\right|\leq4\mathrm{dist}\left(\mathcal{X},\mathcal{Y}\right).\]
Therefore
\[
\left|\sigma_{\mathcal{X}}^{2}(\mathbf{b}_{j})-\sigma_{\mathcal{X}}^{2}(\mathbf{b}_{j})\right|=\left|\sum_{r}\sigma_{X_{r}}^{2}(\mathbf{b}_j)-\sum_{r}\sigma_{Y_{r}}^{2}(\mathbf{b}_j)\right|\leq4d \cdot \mathrm{dist}\left(\mathcal{X},\mathcal{Y}\right),
\]
for all $j$.
\end{proof}

\begin{lem}
\label{lem:commuting-zero-spread} Suppose $\mathcal{Y}=\{Y_{1},\dots,Y_{d}\}$
is a set of commuting Hermitian matices on $\mathbb{C}^{n}$. There
exists an orthonormal basis $\mathcal{B}=\{\mathbf{b}_{1},\dots,\mathbf{b}_{n}\}$
of $\mathbb{C}^{n}$ for which $\mu_{\mathcal{Y}}(\mathcal{B})=0.$ 
\end{lem}

\begin{proof}
We take $\mathcal{B}$ to be an orthonormal basis of common eigenvectors
of the $Y_{r}$ and then notice
\[
\left\langle Y_{r}^{2}\mathbf{b}_{j},\mathbf{b}_{j}\right\rangle -\left\langle Y_{r}\mathbf{b}_{j},\mathbf{b}_{j}\right\rangle ^{2}=0. \qedhere
\]
\end{proof}

\begin{lem}
\label{lem:EffectOfCompression} 
Suppose $W \colon \mathbb{C}^{n_{1}}\rightarrow\mathbb{C}^{n_{2}}$
is an isometry. Let $P=WW^{*}.$ Suppose $\mathcal{X}=\{X_{1},\dots,X_{d}\}$
is a set of Hermitian matices on $\mathbb{C}^{n_{2}}$ with $\left\Vert \mathcal{X}\right\Vert \leq1$
and $\left\Vert \left[X_{r},P\right]\right\Vert \leq\delta$ for all
$r$. Suppose $\mathcal{B}=\{\mathbf{b}_{1},\dots,\mathbf{b}_{k}\}$
is an orthonormal set in $\mathbb{C}^{n_{1}}$. The orthonormal set
$W\mathcal{B}=\{W\mathbf{b}_{1},\dots,W\mathbf{b}_{k}\}$ and the
set $W^{*}\mathcal{X}W=\{W^{*}X_{1}W,\dots,W^{*}X_{d}W\}$ satisfies
\[
\mu_{\mathcal{X}}(W\mathcal{B})\leq\mu_{W^{*}\mathcal{X}W}(\mathcal{B})+8d\delta.
\]
\end{lem}

\begin{proof}
Let $Y_{r}=PX_{r}P+(I-P)X_{r}(I-P)$. Then 
\[
\left\Vert X_{r}-Y_{r}\right\Vert \leq2\left\Vert PX_{r}(I-P)\right\Vert \leq2\left\Vert \left[P,X_{r}\right]\right\Vert 
\]
so $\mathrm{dist}\left(\mathcal{X},\mathcal{Y}\right)\leq2\delta$
and therefore, by Lemma \ref{lem:continuityOfSpread},
\[
\mu_{\mathcal{X}}(W\mathcal{B})\leq\mu_{\mathcal{Y}}(W\mathcal{B})+8d\delta,
\]
where $\mathcal{Y} = \{ Y_1, Y_2, \ldots, Y_d \}$. 
For any $\mathbf{b}$ in $\mathbb{C}^{n_{1}}$ we find 
\[
\left\langle Y_{r}W\mathbf{b},W\mathbf{b}\right\rangle =\left\langle W^{*}Y_{r}W\mathbf{b},\mathbf{b}\right\rangle 
\]
and
\begin{align*}
\left\langle Y_{r}^{2}W\mathbf{b},W\mathbf{b}\right\rangle  & =\left\langle W^{*}Y_{r}^{2}PW\mathbf{b},\mathbf{b}\right\rangle \\
 & =\left\langle W^{*}Y_{r}WW^{*}Y_{r}W\mathbf{b},\mathbf{b}\right\rangle \\
 & =\left\langle \left(W^{*}Y_{r}W\right)^{2}\mathbf{b},\mathbf{b}\right\rangle .
\end{align*}
This implies
\[
\mu_{W^{*}\mathcal{Y}W}(\mathcal{B})=\mu_{\mathcal{Y}}(W\mathcal{B}).
\]
However, $W^{*}Y_{r}W=W^{*}X_{r}W$ so 
\[
\mu_{W^{*}\mathcal{X}W}(\mathcal{B})=\mu_{W^{*}\mathcal{Y}W}(\mathcal{B})=\mu_{\mathcal{Y}}(W\mathcal{B}). \qedhere
\]
\end{proof}

\begin{prop}
Suppose $P$ is a projection in $\mathbf{M}_{n}(\mathbb{C})$ and
$\hat{\mathcal{X}}=\{\hat{X}_{1},\hat{X}_{2},\hat{X}_{3},\hat{X}_{4}\}$
is a representation of $\mathcal{T}_{0}^{\prime}$ and $\|[P,\hat{X}_{r}]\|\leq\delta$
for all $r$. Let $n_{1}$ be the rank of $P$ and suppose 
$W \colon \mathbb{C}^{n_{1}}\rightarrow\mathbb{C}^{n}$
is any isometry with $WW^{*}=P$. Let $X_{r}=W^{*}\hat{X_{r}}W$ and $\mathcal{X} = \{ X_1, X_2, X_3, X_4 \}$. 
\begin{enumerate}
\item $\left\Vert \mathcal{X}\right\Vert \leq1$.
\item $\mathcal{X}=\{X_{1},X_{2},X_{3},X_{4}\}$ is a representation of
$\mathcal{T}_{2\delta}^{\prime}$.
\item If there is a representation $\mathcal{Y}$ of $\mathcal{T}_{0}^{\prime}$
with $\mathrm{dist}(\mathcal{X},\mathcal{Y})\leq\epsilon$ then there
is an orthogonal basis $\mathcal{B}$ of $P\mathbb{C}^{n}$ with 
\[
\mu_{\hat{\mathcal{X}}}(\mathcal{B})\leq8d\delta+4d\epsilon.
\]
\end{enumerate}
\end{prop}

\begin{proof}
(1) The norm condition is easy, as it certainly is true for the $\hat{X}_{r}$. 

(2) For any $r$ and $s$,
\begin{align*}
\left\Vert X_{r}X_{s}-X_{s}X_{r}\right\Vert  & =\left\Vert W^{*}\hat{X}_{r}P\hat{X}_{s}W-W^{*}\hat{X}_{s}P\hat{X}_{r}W\right\Vert \\
 & \leq\left\Vert \hat{X}_{r}P\hat{X}_{s}-\hat{X}_{s}P\hat{X}_{r}\right\Vert \\
 & \leq\left\Vert \hat{X}_{r}P\hat{X}_{s}-\hat{X}_{r}\hat{X}_{s}P\right\Vert +\left\Vert \hat{X}_{s}\hat{X}_{r}P-\hat{X}_{s}P\hat{X}_{r}\right\Vert \\
 & \leq2\delta
\end{align*}
and, for $(r,s)$ equal $(1,2)$ or $(3,4)$, we find
\begin{align*}
 & \left\Vert X_{r}^{2}+X_{s}^{2}-I\right\Vert   \\
 & =\left\Vert W^{*}\hat{X}_{r}P\hat{X}_{r}W+W^{*}\hat{X}_{s}P\hat{X}_{s}W-W^{*}PW\right\Vert \\
 & \leq\left\Vert \hat{X}_{r}P\hat{X}_{r}+\hat{X}_{s}P\hat{X}_{s}-P\right\Vert \\
 & \leq\left\Vert \hat{X}_{r}P\hat{X}_{r}-\hat{X}_{r}^{2}P\right\Vert +\left\Vert \hat{X}_{s}P\hat{X}_{s}-\hat{X}_{s}^{2}P\right\Vert +\left\Vert \hat{X}_{r}^{2}P+\hat{X}_{s}^{2}P-P\right\Vert \\
 & \leq2\delta.
\end{align*}

(3) Follows from (1), (2) and Lemmas \ref{lem:continuityOfSpread}, \ref{lem:commuting-zero-spread}, and \ref{lem:EffectOfCompression}.
\end{proof}

\begin{defn}
If we had made a different choice of partial isometry, $W_{1}$, then
the resulting $4$-tuples would be unitarily equivalent to the first,
since $W^{*}W_{1}$ is a unitary. Therefore, the Bott index,
\[
\mathrm{Bott}(X_{1}+iX_{2},X_{3}+iX_{4})
\]
does not depend on the choice of $W$. We therefore define
\[
\mathrm{Bott}(P;\hat{X}_{1},\hat{X}_{2},\hat{X}_{3},\hat{X}_{4})=\mathrm{Bott}(X_{1}+iX_{2},X_{3}+iX_{4})
\]
so long as $\delta$ is small enough to ensure a spectral gap in the
Bott matrix $B(H_1,H_2,H_3)$.
\end{defn}

In this form, depending on a projection almost commuting with commuting
matrices, our Bott index is essentially the same as the Chern number
as calculated on finite samples, in \cite{prodan2011disordered}.

If the projection $P$ and the $\hat{X}_{r}$ are real, then we can
select the isometry $W$ to be real which means that the $X_{r}$
will be real. 

If the projection $P$ and the $\hat{X}_{r}$ are self-dual in $\mathbf{M}_{2N}(\mathbb{C})$,
we can use the spectral theorem for normal quaternionic matrices to
find an isometry $W$ with $WW^{*}=P$ and $W^{*}=-Z_{N_{1}}W^{\mathrm{T}}Z_{N}$.
This means\[
X_{r}^{\sharp}=-Z_{N_{1}}W^{\mathrm{T}}\hat{X_{r}}\overline{W}Z_{N_{1}}=-W^{*}Z_{N}\hat{X_{r}}Z_{N}W=W^{*}\hat{X_{r}}W,
\]
so the $X_{r}$ are self-dual.
We can unambiguously define
\[
\mathrm{Pf}\!\mathrm{-}\!\mathrm{Bott}(P;\hat{X}_{1},\hat{X}_{2},\hat{X}_{3},\hat{X}_{4})=\mathrm{Pf}\!\mathrm{-}\!\mathrm{Bott}(X_{1}+iX_{2},X_{3}+iX_{4}).
\]

We finish then with corollaries to the theorems in \S \ref{sub:Main-theorems}
about a lattice on a torus. Similar results are true for lattices
on a sphere, but these are less popular in physics.
\begin{thm}
For every $\epsilon>0$, there is a $\delta$ in $(0,\tfrac{1}{8})$
so that, for every $4$-tuple of commuting Hermitian matrices $\hat{X}_{1},\hat{X}_{2},\hat{X}_{3},\hat{X}_{4}$
in $\mathbf{M}_{n}(\mathbb{C})$ with\[
\hat{X}_{1}^{2}+\hat{X}_{2}^{2}=\hat{X}_{3}^{2}+\hat{X}_{4}^{2}=I,\]
 if $P$ is a projection with $\|[P,\hat{X}_{r}]\|\leq\delta$ for
all $r$, and if \[
\mathrm{Bott}(P;\hat{X}_{1},\hat{X}_{2},\hat{X}_{3},\hat{X}_{4})=0,\]
 then there there is an orthonormal basis $\mathbf{b}_{1},\dots,\mathbf{b}_{n_{1}}$
of the subspace $P\mathbb{C}^{n}$ so that\[
\left\langle \hat{X}_{r}^{2}\mathbf{b}_{j},\mathbf{b}_{j}\right\rangle -\left\langle \hat{X}_{r}\mathbf{b}_{j},\mathbf{b}_{j}\right\rangle ^{2}\leq\epsilon\]
 for all $r$ and all $j$.
\end{thm}

\begin{thm}
For every $\epsilon>0$, there is a $\delta > 0$ 
so that, for every $4$-tuple of commuting Hermitian matrices $\hat{X}_{1},\hat{X}_{2},\hat{X}_{3},\hat{X}_{4}$
in $\mathbf{M}_{n}(\mathbb{R})$ with\[
\hat{X}_{1}^{2}+\hat{X}_{2}^{2}=\hat{X}_{3}^{2}+\hat{X}_{4}^{2}=I,\]
 if $P$ is a real projection with $\|[P,\hat{X}_{r}]\|\leq\delta$
for all $r$, then there there is an orthonomal basis $\mathbf{b}_{1},\dots,\mathbf{b}_{n_{1}}$
of the subspace $P\mathbb{R}^{n}$ so that\[
\left\langle \hat{X}_{r}^{2}\mathbf{b}_{j},\mathbf{b}_{j}\right\rangle -\left\langle \hat{X}_{r}\mathbf{b}_{j},\mathbf{b}_{j}\right\rangle ^{2}\leq\epsilon\]
 for all $r$ and $j$.
\end{thm}

\begin{thm}
For every $\epsilon>0$, there is a $\delta$ in $(0,\tfrac{1}{8})$
so that, for every $4$-tuple of commuting Hermitian, self-dual matrices
$\hat{X}_{1},\hat{X}_{2},\hat{X}_{3},\hat{X}_{4}$ in $\mathbf{M}_{2N}(\mathbb{C})$
with\[
\hat{X}_{1}^{2}+\hat{X}_{2}^{2}=\hat{X}_{3}^{2}+\hat{X}_{4}^{2}=I,\]
 if $P$ is a self-dual projection with $\|[P,\hat{X}_{r}]\|\leq\delta$
for all $r$, and if \[
\mathrm{Pf}\!\mathrm{-}\!\mathrm{Bott}(P;\hat{X}_{1},\hat{X}_{2},\hat{X}_{3},\hat{X}_{4})=1,\]
 then there are vectors $\mathbf{b}_{1},\dots,\mathbf{b}_{N_{1}}$
so that \[
\mathbf{b}_{1},\dots,\mathbf{b}_{N_{1}},\mathcal{T}\mathbf{b}_{1},\dots,\mathcal{T}\mathbf{b}_{N_{1}}\]
is an orthonormal basis of the subspace $P\mathbb{C}^{2N}$ and\[
\left\langle \hat{X}_{r}^{2}\mathbf{b}_{j},\mathbf{b}_{j}\right\rangle -\left\langle \hat{X}_{r}\mathbf{b}_{j},\mathbf{b}_{j}\right\rangle ^{2}\leq\epsilon\]
 for all $r$ and all $j$.
\end{thm}

\section{Discussion}

Our results on localization have very minimal assumptions. We do not
require that the projection $P$ arises from a gap in the
specrtrum of the Hamiltonian. We also allow that the distribution
of sites within the torus can be very irregular. We pay for this with
very weak localization in the conclusion. In contrast, one can study
specific models and find exponential decay in a basis made from time-reversal
pairs, see \cite{soluyanov2011wannier}. 

Our results might have applications in signal processing in areas such
as hypercomplex processes  \cite{took2009quaternion} and blind source
separation \cite{cardosoSimultanDiagn,waxJointDiag}.

\section*{Acknowledgments}

\thanks{We thank Matthew Hastings for his suggestions regarding the present
article and for his role in laying out a program relating the theory
of almost commuting matrices to condensed matter physics.

This research was supported by the Danish National Research Foundation (DNRF)
through the Centre for Symmetry and Deformation.}

\end{document}